\documentclass[12pt,reqno]{amsart}
\usepackage[margin=1.2in,letterpaper]{geometry}

\usepackage{graphicx}
\usepackage{amssymb}
\usepackage{amsthm}
\usepackage{mathrsfs}
\usepackage{stmaryrd}
\usepackage{accents} 
\usepackage{enumitem} 
\usepackage{colonequals}
\usepackage{subcaption}
\usepackage{microtype}
\usepackage[bookmarksopen,bookmarksdepth=2]{hyperref} 
\allowdisplaybreaks
\newcommand{\be}{\begin{equation} }
\newcommand{\ee}{\end{equation}}
\newcommand{\bse}{\begin{subequations}}
\newcommand{\ese}{\end{subequations}}

\newcommand{\supp}[1]{\operatorname{supp}{#1}}

\newcommand{\LV}{\left|}
\newcommand{\RV}{\right|}
\newcommand{\LN}{\left\|}
\newcommand{\RN}{\right\|}
\newcommand{\LB}{\left[}
\newcommand{\RB}{\right]}
\newcommand{\LC}{\left(}
\newcommand{\RC}{\right)}
\newcommand{\LA}{\left<}
\newcommand{\RA}{\right>}
\newcommand{\LCB}{\left\{}
\newcommand{\RCB}{\right\}}
\newcommand{\p}{\partial}
\newcommand{\ep}{\varepsilon}
\newcommand{\R}{\mathbb{R}} 
 
\newcommand{\Z}{\mathbb{Z}} 
\newcommand{\T}{\mathbb{T}}

\newcommand{\ev}{\Lambda}
\newcommand{\amp}{a^{\pm}_{n+1}}
\newcommand{\err}{\mathbf{Err}}

\newcommand{\boldb}{\boldsymbol{b}}
\newcommand{\boldd}{\boldsymbol{d}}
\newcommand{\boldB}{\boldsymbol{B}}
\newcommand{\boldD}{\boldsymbol{D}}
\newcommand{\boldu}{\boldsymbol{u}}
\newcommand{\boldv}{\boldsymbol{v}}
\newcommand{\boldW}{\boldsymbol{W}}
\newcommand{\rev}{\boldsymbol{r}}
\newcommand{\lev}{\boldsymbol{\ell}}
\newcommand{\rem}{\boldsymbol{R}}
\newcommand{\id}{\textrm{I}_2}

\theoremstyle{plain} 
\newtheorem{theorem}{Theorem}[section]

\newtheorem{lemma}{Lemma}[section]
\newtheorem{proposition}{Proposition}[section]
\newtheorem{definition}{Definition}[section]
\theoremstyle{remark} 
\newtheorem{remark}{Remark}[section]

\theoremstyle{definition}

\numberwithin{equation}{section}

\title[$C^0$-nonuniqueness for 1D Conservation laws]{Nonuniqueness for continuous solutions to 1D hyperbolic systems} 

\subjclass[2010]{35L45, 35L65, 76N10}
\keywords{Global weak solutions, non-uniqueness,  Conservation laws, Liu conditions}

\author[R.M. Chen]{Robin Ming Chen}
\address{ Department of Mathematics,
      University of Pittsburgh.}
\email{mingchen@pitt.edu}
\author[A.F. Vasseur]{Alexis F. Vasseur}
\address{Department of Mathematics,
The University of Texas at Austin.}
\email{vasseur@math.utexas.edu}
\author[C. Yu]{Cheng Yu}
\address{Department of Mathematics,
University of Florida.}
\email{chengyu@ufl.edu}

\date{}       
\begin{document}
\begin{abstract}
In this paper, we show that   a geometrical condition on 2 × 2 systems of conservation laws leads to  non-uniqueness  in the class of 1D continuous functions. This demonstrates that the Liu Entropy Condition alone is insufficient to guarantee uniqueness, even within the mono-dimensional setting. We provide examples of systems where this pathology holds, even if they verify stability and uniqueness for small BV solutions. Our proof is based on the convex integration process. Notably, this result represents the first application of convex integration to construct non-unique continuous solutions in one dimension.

\end{abstract}

\thispagestyle{empty}
\maketitle

\setcounter{tocdepth}{1}
\tableofcontents

\section{Introduction}
The aim of this paper is to describe non-uniqueness pathologies for continuous solutions to mono-dimensional conservation laws.
We are considering $2\times2$ hyperbolic systems of conservation laws in one space dimension:
\begin{equation}\label{eq cons}
\p_t \boldu + \p_x f(\boldu) = 0 \qquad \text{for } (t, x) \in \R_+ \times \T,
\end{equation}
where $\T$ is the one-dimensional torus $[0,1]$. The flux function $f:\mathcal{V}\to \R^2$ is a $C^\infty$ function defined on a neighborhood of the origin $\mathcal{V}\subset \R^2$.  For all $\boldu\in \mathcal{V}$, the system is  strictly hyperbolic, when  the Jacobian matrix $Df(\boldu)$ has two distinct real eigenvalues: $\ev^-(\boldu)<\ev^+(\boldu)$.  
\vskip0.3cm

The study of hyperbolic systems of conservation laws has its roots in the work of Riemann in 1860, where he investigated the isentropic gas dynamics. 
For such $2\times 2$ systems, it is possible to construct global uniformly bounded solutions for general initial values, using the compensated compactness method \cite{Tartar}.
However, the problem of uniqueness in this class is completely open. 

\vskip0.3cm
A fundamental difficulty to study the uniqueness of  such systems is the development of  discontinuities  in finite time, known as shocks. This motivates the introduction of additional admissibility conditions. The prevailing view is that for conservation laws in dimension 1, the issue of admissibility for general weak solutions should be resolved through a test applied to every point of the shock set of the solutions (see Dafermos \cite{Dafermos} Chapter 8, page 205).  
\vskip0.3cm
This has been proved to be correct in the small BV framework.  Bressan and De Lellis  proved in \cite{BDle}  the uniqueness of small BV solutions under the only assumption that all points of approximate jump satisfy the Liu admissibility conditions \cite{Liu}. 

However, we show in this article  that  the uniqueness of weak solutions cannot be enforced that way in general. For a family of systems \eqref{eq cons},  we construct non-unique solutions which do not have any discontinuities.

\vskip0.3cm

Since the system is strictly hyperbolic, the spectral gap at 0 is positive:  
$$
\delta_\ev := \ev^+(0)-\ev^-(0)>0.
$$
Moreover, we can choose a base of  right eigenvectors of $Df(\boldu)$,  $\left\{\rev_i(\boldu), i=\pm\right\}$, defined as regular functions of $\boldu$ on $\mathcal{V}$. We have 
$$
A := \left|\mathrm{det} (\rev_-(0), \rev_+(0))\right|>0.
$$
Consider the integral curves of these vector fields passing through the origin:
$$
\frac{d \boldu_i(s)}{ds} = \rev_i(\boldu_i(s)), \qquad \boldu_i(0)=0.
$$
We denote $\kappa_i$, $i = \pm$, the curvature of these curves at 0. Our condition on System \eqref{eq cons} to exhibit non-uniqueness pathologies is the following.
\begin{definition}\label{cond}
For any given $0<\ep<1$, we say that that the system \eqref{eq cons} verifies the condition $\mathcal{C}_\ep$ if $\kappa_->0, \kappa_+>0$, and:
\begin{equation}\label{eq cond} \tag{$\mathcal{C}_\ep$}
\LV (\nabla \ev^-\cdot \rev_-)(0) \RV \leq \ep\frac{\kappa_+ \delta_\ev}{A},\qquad \LV (\nabla \ev^+\cdot \rev_+)(0) \RV \leq \ep\frac{\kappa_- \delta_\ev}{A}.
\end{equation}
\end{definition} 

\subsection{Main result}
Under this condition, we can show the following main theorem.
\begin{theorem}\label{thetheo}
There exists $\ep>0$ such that for any system \eqref{eq cons} verifying the condition $\mathcal{C}_{\ep}$ the following holds true. There exists $\eta >0$ such that for any ball $B\subset B(0,\eta)$, we can find at least  two global weak solutions in $C^0(\R^+\times\T; B)$ of \eqref{eq cons} with the same initial value. 
\end{theorem}

To be more precise, we define a weak solution in the following sense.
\begin{definition}
\label{definition of weak solutions}
A bounded measurable function $\boldu(t,x)$ is called a weak solution of \eqref{eq cons} with the bounded and measurable intial data $\boldu_0$, provided that the following equality holds for all $\varphi\in C^1_0(\R\times\T)$:
\begin{equation}
\label{weak formulation}
\int_0^t\int_{\mathbb{T}}(\boldu \varphi_t+f(\boldu)\varphi_x)\,dx\,dt+\int_{\mathbb{T}}\boldu_0\varphi(x,0)\,dx=0.
\end{equation}
\end{definition}
For the sake of clarity,  we focus in this article on the construction of only two different solutions. However, our proof can be easily extended to obtain infinitely many such solutions. 
\vskip0.3cm
\begin{remark}Note that the result is not true in the scalar case. Indeed, any continuous solution $u\in C^0(\R^+\times \T; \R)$ of a scalar conservation laws of the form \eqref{eq cons} is unique. This is a consequence of the uniqueness in $C^1$  of solutions to the associated Hamilton-Jacobi equation \cite{CL1}. Consider $v(t,x)=\int_0^xu(t,y)\,dy-\int_0^tf(u(s,0))\,ds$. Then $v\in C^1(\R^+\times\R;\R)$ is the unique solution to the Hamilton-Jacobi equation 
$$
\partial_t v+f(\partial_x v)=0.
$$
\end{remark}

\begin{remark}
Our result is then optimal in terms of space dimension ($d=1$) and size of the systems ($2\times 2$). 
Note that a  similar result was proved by Giri and Kwon \cite{Giri-Kwon} in dimension bigger than 2 for the isentropic Euler system. Theorem \ref{thetheo} however is the first 1D result of non-uniqueness for continuous solutions to conservation laws. 

\end{remark}

The condition of positive curvatures excludes the cases of \emph{linear fluxes} or \emph{trivial systems} formed of two independent scalar conservation laws, since in these cases, the integral curves would be lines. This prevents also the case of Rich systems which share a lot of properties with the scalar case. 
\vskip0.3cm
Theorem \ref{thetheo} offers a strikingly different picture with what is known in the small BV theory. Extensive efforts have been devoted to this case, employing various methods such as the Glimm scheme, front tracking scheme, and vanishing viscosity method (see for instance \cite{Dafermos, Br} for a survey). These approaches have been instrumental in the thorough investigation of the well-posedness of small BV solutions to systems. The uniqueness of solutions in this framework has been  developed by  Bressan and al in the late 90' \cite{BLe, BL} (See also Liu and Yang \cite{LiuYang}).  Technical conditions have been removed recently in \cite{BG,BDle}.
Note that all these works proved the uniqueness and $L^1$ stability of small BV solutions among solutions from the same class of regularity. 
\vskip0.3cm
In the last decade, the method of $a$-contraction with shifts   in \cite{CKV,GKV} extended those  results to weak/BV uniqueness and stability results (in the spirit of weak/strong principles of Dafermos and DiPerna \cite{Dafermos1, DiPerna}). Considering cases with a strictly convex entropy functional, it shows that small BV solutions are unique among a large class of entropic weak solutions (bounded and verifying the so-called very strong trace property). 
\vskip0.3cm
Bianchini and Bressan showed in \cite{BB}, that in the case of artificial viscosity,  the unique BV solution  can be obtained and selected via the inviscid limit. 
In the isentropic case, the result was extended to inviscid limit of the  Navier-Stokes equation in \cite{CKV2} (see also \cite{KV} and \cite{V}).   This result, based on the $a$-contraction theory, extends also the uniqueness and stability of small BV solutions among the large class of any inviscid limits of the Navier-Stokes equation.   
\vskip0.3cm
It would be interesting to see if either the use of a convex entropy, or the principle of inviscid limit could restore uniqueness in our setting.

\vskip0.3cm
Our method is  based on convex integration first introduced by De Lellis and Szekelyhidi \cite{de2009euler, de2010euler}  to show  non-uniqueness results for  the incompressible Euler. 
For compressible fluid, convex integration was used for the first time by Chiodaroli, De Lellis, and Kreml \cite{CDK} to demonstrate the non-uniqueness of weak solutions to the isentropic compressible Euler system with Riemann initial data in 2D. Recently, Giri and Kwon constructed non-unique continuous entropic solutions also in 2D in \cite{Giri-Kwon}. Their primary method involves the convex integration technique developed for the incompressible Euler equations. This approach, however, cannot be extended directly to hyperbolic systems of conservation laws in 1D, since 1D incompressible flows are trivial.
\vskip0.3cm
In a different approach, Krupa and Szekelyhidi  investigated the  non-uniqueness  for 1D (possibly) discontinuous entropic solutions  in \cite{krupa2022nonexistence}. they showed that the classical T4 convex integration method cannot be applied in this context (see also Lorent and Peng \cite{Lorent-Peng}, and Johansson and Tione \cite{Johansson_2023} for the $p$-system). Finally, Krupa showed in \cite{K} that without entropy condition, it is possible to construct solutions of the $p$-system that are so oscillating that they do not even verify the Rankine-Hugoniot condition.  
\vskip0.3cm
In order to construct non-unique continuous solutions, we are developing new techniques that amplify oscillations in line with the strict  hyperbolic feature. We will explain our main idea in Section 2.

\vskip0.3cm
\subsection{Comment on Condition \ref{eq cond}}
Along the integral curve $\boldu_i$, the quantity
$$
(\nabla \ev^i\cdot \rev_i)(\boldu (s))=\frac{d\ev^i(\boldu(s))}{ds}
$$
is the rate of change of the $i$-th eigenvalue along the integral curve. Therefore, the condition \ref{eq cond} of Definition \ref{cond} illustrates that for each characteristic field, the rate of change of the associated characteristic speed at 0 in the direction of the corresponding eigenvectors is very small compared to the ratio between the product of the curvature of the other integral curve and the spectral gap at 0, and the ``area distortion'' induced by two normalized eigenvectors.
\vskip0.3cm
In the theory of conservation laws, an $i$-th  characteristic field is called \emph{linearly degenerate} if $\nabla\ev^i\cdot \rev_i$ is equal to 0 in $\mathcal{V}$, and it is called \emph{genuinely nonlinear} if   $\nabla\ev^i\cdot \rev_i\neq 0$ in $\mathcal{V}$. 
If both characteristic fields are genuinely nonlinear we say the system is a genuinely nonlinear system.
Note that the condition \ref{eq cond} is always verified for linearly degenerate fields with non-zero curvatures. 
\vskip0.3cm

\subsection{Example}To illustrate our Theorem \ref{thetheo}, we  consider the following  system:
\begin{equation}\label{example}
\left\{
\begin{split}
& \partial_t u + \p_x \LC \frac{u v}{2}+v \RC = 0, \\
& \partial_t  v + \p_x \LC u-\frac{v^2}{2} \RC = 0.
\end{split}\right.
\end{equation}
We show the following theorem.
\begin{theorem}\label{theo-ex}
There exists   $\mathcal{V}\subset \R^2$,  such that both characteristic fields of \eqref{example} are genuinely nonlinear in $\mathcal{V}$. Moreover, for any ball $B\subset \mathcal{V}$ there exist at least  two weak  solutions of \eqref{example} in $C^0(\R^+\times\T;B)$ with the same initial value. 
\end{theorem}
Genuinely nonlinear fields are the natural extensions to systems of convex flux for scalar conservation laws. This example shows that non-uniqueness for continuous weak solutions can hold even under these conditions. 
\begin{remark}
In the 70', Glimm and Lax  constructed in \cite{GL} solutions  to general $2\times2$ genuinely nonlinear systems, for any small enough initial data in $L^\infty$ (see also Bianchini, Colombo, Monti \cite{BCM}).  Our result shows that some of these solutions are not unique in the class of solutions verifying the Liu condition.
\end{remark}
\vskip0.3cm

\vskip0.3cm
The rest of the paper is structured as follows. We give the main idea of the proof in Section \ref{sec:ideas}. We describe the notion of subsolutions and the approximation scheme in Section \ref{sec subsoln approx}. The strength of the high frequency waves is introduced in Section \ref{sec amp}. We describe the induction argument and prove the convergence in Section \ref{sec induct}. The non-uniqueness through the dephasing process is done in Section \ref{sec nonunique}, then our main Theorem \ref{thetheo} follows.  Finally, System \eqref{example} is 
studied in Section \ref{sec example}.

\section{Ideas of the proof}\label{sec:ideas}

The goal of this paper is to show that under the condition $\mathcal{C}_{\ep}$ of Definition \ref{cond} for $\ep>0$ small enough, for any ball $B$ in a small neighborhood of 0, System \eqref{eq cons} admits multiple continuous solutions $\boldu \in C^0(\R_+ \times \T; B)$ sharing the same initial data.
\vskip0.1cm
Since we assume that the system \eqref{eq cons} is regular and strictly hyperbolic, we have:
\begin{enumerate}[label=\rm(H\arabic*)]
\item \label{cond f 1} $f : \R^2 \cap B_r(0) \to \R^2$ and $f \in C^\infty(B_r(0))$ for some $r > 0$.
\item \label{cond f 2} $Df(0)$ has two distinct real eigenvalues $\ev^\pm(0)$, with the associated (normalized) right eigenvectors $\rev_{\pm}(0)$. We denote
\begin{equation}\label{eigenvalue inner product}
p_0 := \LA \rev_+ (0), \rev_- (0) \RA.
\end{equation}
Strict hyperbolicity implies that $0 \le p_0 < 1$.
\end{enumerate}
\vskip0.3cm
We denote $\lev_\pm(0)$  the left eigenvectors of $Df(0)$ corresponding to $\ev^\pm(0)$ respectively, and  
\begin{equation}\label{eq def b and d}
\begin{split}
 \boldb_+ & := D^2f(0) : \LC \rev_+(0) \otimes \rev_+(0) \RC, \quad \boldb_- := D^2f(0) : \LC \rev_-(0) \otimes \rev_-(0) \RC, \\
 \boldsymbol{d} & := D^2f(0) : \LC \rev_+(0) \otimes \rev_-(0) \RC.
\end{split}
\end{equation}
\vskip0.3cm
Following the general methodology of convex integration, we will construct a family of approximations $\{ (\boldu_n, E_{n,-}, E_{n,+}) \}$ with $E_{n, \pm} > 0$ such that
$$
\p_t \boldu_n + \p_x \LB f(\boldu_n) + E_{n,-} \boldb_- + E_{n,+} \boldb_+ \RB = 0.
$$
This is an approximation to the system \eqref{eq cons} where the error term (equivalent to the Reynolds tensor in the classical convex integration of the incompressible Euler equations) is projected on the basis  $(\boldb_-,\boldb_+)$ as defined in \eqref{eq def b and d}. Note that Lemma \ref{lemm structure} below will actually prove that under the Hypothesis of Definition \ref{cond}, this forms a basis of $\R^2$. 
The rough idea is then to construct recursively the family $\boldu_n$ by adding highly oscillating functions $\boldv_{n+1}$
$$
\boldu_{n+1} = \boldu_n + \boldv_{n+1},
$$
such that $\boldu_n$ converges in $C^0$, and that the error terms $E_{n,-}$ and $E_{n,+}$ converge in a controlled way to 0. Adding  phase shifts in the oscillations of   the functions $\boldv_{n+1}$ ensures that we can obtain different solutions at the limit. The correction term $\boldv_{n+1}$ has actually  two parts: $\boldv_{n+1} = \boldv^1_{n+1} + \boldv^2_{n+1}$. Let us first focus on the first level of correction $ \boldv^1_{n+1} $. 
\vskip0.3cm
A careful reader may notice that we are lightly oversimplifying the argument here, since the oscillating function is actually added to a slightly regularized $\boldu_n$ (see \eqref{eq est mean osc}). This slight regularization is for technical reasons which are classical in the convex integration method. It allows a sharp control on higher derivatives of $\boldu_n$ which is needed during the expansion.
\vskip0.3cm
The computation involves an expansion of the flux function near 0. The correction of the error term is done at the order 2. Because of that it is very important to carefully tune the oscillations in the eigen-modes of $Df(\boldu_n)$. (In the parlance of convex integration for the incompressible Euler, it is to avoid as much as possible transport and Nash errors).  
The rough idea is to construct a first level of correction $\boldv^1_{n+1}$ as 
\begin{eqnarray*}
\boldv^1_{n+1}(t,x)&=& \p_x \Big\{ a^+_{n+1}(t,x) \rev_+(\boldu_n(t,x))\sin(\lambda_{n+1}(x-\Lambda ^+(\boldu_n(t,x))t))\\
&& \qquad + a^-_{n+1}(t,x)\rev_-(\boldu_n(t,x)) \sin(\lambda_{n+1}(x-\Lambda ^-(\boldu_n(t,x))t)) \Big\},
\end{eqnarray*}
where the wave amplitude $a^\pm_{n+1}(t,x)$, the right eigenvectors $\rev_\pm(\boldu_n(t,x))$ of $Df(\boldu_n)$, and the eigenvalues $\Lambda ^\pm (\boldu_n(t,x))$ can be seen as low frequency with respect to the new high frequency $\lambda_{n+1}$. (Actually, the oscillations of $\Lambda ^\pm (\boldu_n(t,x))$ are too fast, necessitating the localization of phase in Subsection \ref{subsec loc}). Then, taking into account only the high order oscillations and the first term of correction, we have roughly for $\tilde\boldu_{n+1}=\boldu_n+\boldv^1_{n+1}$, up to small errors denoted by $\mathbf{Err}$, that
$$
\partial_t \tilde\boldu_{n+1}+\partial_x\left[ f(\boldu_n)+Df(\boldu_n)\boldv^1_{n+1} +  E_{n,-} \boldb_- + E_{n,+} \boldb_++\mathbf{Err}\right]=0.
$$
And so, up to possible additional errors from truncating the expansion of the flux function $f$ at the second order (we still denote $\mathbf{Err}$ the cumulative error):
\begin{equation}\label{eq approx}
\partial_t \tilde\boldu_{n+1}+\partial_x\left[ f(\tilde\boldu_{n+1})-\frac{D^2f(\boldu_n)}{2}:(\boldv^1_{n+1}\otimes \boldv^1_{n+1})+  E_{n,-} \boldb_- + E_{n,+} \boldb_++\mathbf{Err}\right]=0.
\end{equation}
Using that 
\[
\sin^2(y)=1/2-\cos(2y)/2$$ and $$2\sin(y)\sin(z)= \cos(y-z)-\cos(y+z),
\]
we have
\be\label{eq still osc}
\begin{split}
\frac{D^2f(\boldu_n)}{2}:(\boldv^1_{n+1}\otimes \boldv^1_{n+1})&=\frac{1}{4}\left(|a^-_{n+1}|^2 \boldb_-+ |a^+_{n+1}|^2 \boldb_+\right)
\\&+(\text{still oscillating terms}).
\end{split}
\ee
Choosing carefully $a^+_{n+1}$ and $a^-_{n+1}$, we can deplete geometrically the error terms $E_{n,-}$ and $E_{n,+}$ when $n$ converges to infinity. 
Note that the other error terms $\mathbf{Err}$  always can be projected onto the basis $( \boldb_-, \boldb_+)$. And because the system is strictly hyperbolic close to 0, we always have two directions of oscillations. However, the remaining oscillating terms are not necessarily small in $L^\infty$ and they pose serious challenges. In the classical theory of convex integration for the incompressible Euler equations, these terms can be absorbed into the pressure. But we don't have this luxury here. We need a principle to filter these oscillations out of the system. This is where the hypothesis based on Definition \ref{cond} comes into play.
\vskip0.3cm 
For the sake of a simple presentation of the idea, let us for now drop the cross terms involved in the still oscillating terms in \eqref{eq still osc} (they are easier to treat anyway). The two other terms are exactly:
\begin{equation}\label{eq still osc main}
\begin{split}
\mu^-_n(t,x) \boldb_-+\mu^+_n(t,x) \boldb_+
& = \frac{1}{4}\left[ |a^-_{n+1}|^2 \boldb_-\cos(2\lambda_{n+1}(x-\Lambda^-t))\right]\\
& \qquad +\frac{1}{4}\left[ |a^+_{n+1}|^2 \boldb_+\cos(2\lambda_{n+1}(x-\Lambda^+t))\right].
\end{split}
\ee
To filter out these oscillations,  we consider the second family of correctors $\boldv^2_{n+1}$ of the form
\begin{eqnarray*}
\boldv^2_{n+1}=
&& \frac{1}{4}\left[ |a^-_{n+1}|^2 \boldB_-\cos(2\lambda_{n+1}(x-\Lambda^-t))\right] \\
&&\qquad +\frac{1}{4}\left[ |a^+_{n+1}|^2 \boldB_+\cos(2\lambda_{n+1}(x-\Lambda^+t))\right]
\end{eqnarray*}
for some suitably chosen $\boldB_\pm$. Then, taking into account only the high order oscillations again:
\be\label{eq second kill osc}
\begin{split}
 &\partial_t \boldv^2_{n+1}+\partial_x(Df(\boldu_n) \boldv^2_{n+1}) \\
 & \ + \partial_x \Big[ \mu^-_n (\Lambda^-\id-Df(\boldu_n))\boldB_-) +\mu^+_n ( \Lambda^+\id-Df(\boldu_n))\boldB_+)+\mathbf{Err} \Big] = 0,
\end{split}
\ee
where $\id$ is the $2\times 2$ identity matrix.

\vskip0.3cm

 Note that these terms in $\boldv^2_{n+1}$ are small compared to $\boldv^1_{n+1}$ (because they are quadratic in amplitude). Therefore the second order error in the expansion of $f$ for this term in \eqref{eq approx} is very small. For the same reason, and because we are constructing very small solutions $\boldu_n \approx 0$, the corrector $\boldv^2_{n+1}$ can help cancel the terms in \eqref{eq still osc main} if we can find vectors $\boldB_-, \boldB_+$ such that 
\begin{eqnarray*}
&& (\ev^-(0) \id-Df(0))\boldB_-=\boldb_-,\\
&& (\ev^+(0) \id-Df(0))\boldB_+=\boldb_+.
\end{eqnarray*}
Multiplying on the left the first equation by the vector $\lev_-(0)$, and the second equation by the vector $\lev_+(0)$, this leads to the condition
$$
\lev_\pm(0) \cdot \boldb_\pm=0.
$$
Note that this ``twisted" condition is equivalent to saying that $\boldb_\pm\in \mathrm{span}\{\rev_\mp(0)\}$.
We do not need such a strong condition, but we need that the component of $\boldb_\pm$ along $\rev_\pm(0)$, $(\boldb_\pm\cdot \lev_\pm(0) )  \rev_\pm(0)$, contributes only a small error when reprojected on the basis $(\boldb_-, \boldb_+)$. This property follows from the assumptions of Definition \ref{cond} for $\ep$ small enough:
\begin{lemma}\label{lemm structure}
For  $0<\ep<1$, if the system \eqref{eq cons} verifies the condition \ref{eq cond} of Definition \ref{cond} then:
$$
\det (\boldb_-, \boldb_+) \ne 0,
$$
and  
$$
(\boldb_\pm\cdot \lev_\pm(0) )  \rev_\pm(0) =\alpha^\pm \boldb_\pm+\beta^\pm \boldb_\mp,
$$
with 
\begin{equation}\label{cond on f3}
|\alpha^\pm|+|\beta^\pm|\leq \frac{\ep}{1-\ep}.
\end{equation}
\end{lemma}

\begin{proof} 
We split the proof into two steps. For simplicity of the presentation, we wite $\rev_\pm = \rev_\pm(0)$ and $\lev_\pm = \lev_\pm(0)$.
\vskip0.1cm \noindent{Step 1. Projection of the vector $\boldb_\pm$ onto the basis $(\rev_+,\rev_-)$.} First, we have
\[
\boldb_\pm=(\boldsymbol{\ell}_\pm\cdot \boldb_\pm) \rev_\pm+(\boldsymbol{\ell}_\mp\cdot \boldb_\pm) \rev_\mp,
\]
where the left eigenvectors are chosen in a way such that $\lev_\pm \cdot \rev_\pm = 1$, and so 
\[
|\boldsymbol{\ell}_\pm|=\frac{1}{|\det(\rev_\pm,\rev_\mp)|} = \frac1A.
\]
We have to compute $(\boldsymbol{\ell}_\pm\cdot \boldb_\pm)$ and $(\boldsymbol{\ell}_\mp\cdot \boldb_\pm)$.
For $\boldu$ in a neighborhood of 0, we have
$$
\boldsymbol{\ell}_\pm(\boldu)[Df(\boldu)- \Lambda^\pm(\boldu) \id ]\rev_\pm(\boldu)=0.
$$
Differentiating in the direction $\rev_\pm(\boldu)$, and evaluating the result at $\boldu=0$, we find
$$
\boldsymbol{\ell}_\pm(0)\cdot [D^2f(0) - \nabla \Lambda^\pm(0) \id ]:(\rev_\pm(0)\otimes\rev_\pm(0))=0,
$$
and so 
$$
\boldsymbol{\ell}_\pm\cdot\boldb_\pm=\rev_\pm \cdot\nabla \Lambda^\pm(0).
$$
In the same way, we have
$$
\boldsymbol{\ell}_\mp(\boldu)[Df(\boldu)-\id\Lambda^\pm(\boldu)]\rev_\pm(\boldu)=0.
$$
Differentiating again in the direction  $\rev_\pm(\boldu)$, and evaluating the result at $\boldu=0$, we find
$$
(\rev_\pm\cdot\nabla)\rev_\pm\cdot\boldsymbol{\ell}_\mp (\Lambda^\mp-\Lambda^\pm)+\boldsymbol{\ell}_\mp \cdot D^2f(0):(\rev_\pm \otimes\rev_\pm )=0.
$$
Since 
$$
(\rev_\pm\cdot\nabla)\rev_\pm\cdot\boldsymbol{\ell}_\mp=|\boldsymbol{\ell}_\mp| \kappa_\pm= \frac{\kappa_\pm}{A},
$$
we find 
$$
\boldsymbol{\ell}_\mp\cdot \boldb_\pm = \pm \frac{\kappa_\pm \delta _\Lambda}{A}.
$$
Therefore
\begin{equation}\label{eq b}
\boldb_\pm=(\rev_\pm\cdot\nabla \Lambda^\pm) \rev_\pm \pm (\kappa_\pm \delta_\Lambda/A) \rev_\mp.
\end{equation}
\vskip0.1cm \noindent{Step 2. Writing $(\rev^\pm,\rev^\mp)$ in the base of $(\boldb^\pm,\boldb^\mp)$.} Inverting the matrix, we find:
\begin{eqnarray*}
&& |\alpha^\pm| = \left|\frac{(\rev_\pm\cdot\nabla \Lambda^\pm)(\rev_\mp\cdot\nabla \Lambda^\mp)}{(\kappa_\pm \delta_\Lambda/A)(\kappa_\mp \delta _\Lambda/A) + (\rev_\pm\cdot\nabla \Lambda^\pm)(\rev_\mp\cdot\nabla \Lambda^\mp)}\right|,\\
&&|\beta^\pm| = \left|\frac{(\rev_\pm\cdot\nabla \Lambda^\pm) (\kappa_\mp \delta _\Lambda/A)}{(\kappa_\pm \delta _\Lambda/A)(\kappa_\mp \delta _\Lambda/A) + (\rev_\pm\cdot\nabla \Lambda^\pm)(\rev_\mp\cdot\nabla \Lambda^\mp)}\right|.
\end{eqnarray*}
Using the estimates of Definition \ref{cond}, we find
$$
 |\alpha^\pm|\leq \frac{\ep^2}{1 - \ep^2}, \qquad
|\beta^\pm|\leq \frac{\ep}{1-\ep^2},
$$
which leads to \eqref{cond on f3}.
\end{proof}

Now we can apply the above lemma to the second-order corrector $\boldv^2_{n+1}$ to help filter out the oscillation in \eqref{eq still osc main}. Note that now
\be\label{eq decomp b}
\boldb_\pm = \LC \alpha^\pm \boldb_\pm + \beta^\pm \boldb_\mp \RC \pm \frac{\kappa_\pm \delta_\ev}{A} \rev_\mp,
\ee
and we can find vectors $\boldB_\pm$ such that
\be\label{eq exist B}
\LB \ev^+ \id - Df(0) \RB \boldB_\pm = \pm \frac{\kappa_\pm \delta_\ev}{A} \rev_\mp =: \tilde{\boldb}_\pm.
\ee
Therefore from \eqref{eq still osc main} and \eqref{eq second kill osc} we find that after applying the corrector $\boldv^2_{n+1}$, the remaining oscillation in \eqref{eq still osc main} becomes
\[
\LC \mu^-_n \alpha^+ + \mu^+_n \beta^- \RC \boldb_+ + \LC \mu^-_n \beta^- + \mu^+_n \alpha^- \RC \boldb_-,
\]
where from \eqref{cond on f3} we have 
\[
\LV \mu^-_n \alpha^+ + \mu^+_n \beta^- \RV, \ \LV \mu^-_n \beta^- + \mu^+_n \alpha^- \RV \le \frac{\ep}{4(1 - \ep)} \LC |a^-_{n+1}|^2 + |a^+_{n+1}|^2 \RC.
\]
Hence this remaining oscillation is much smaller compared with the ``error-depleting'' term \eqref{eq still osc}.

\section{Subsolutions and approximation scheme}\label{sec subsoln approx}

\subsection{Subsolutions}
We start with a relaxed version of \eqref{eq cons} and consider the following notion of {\it subsolutions}.
\begin{definition}\label{def subsoln}
A subsolution to \eqref{eq cons} is a triple $(\boldu_s, E_{s,-}, E_{s,+})$ with $\boldu_s \in C^\infty((0, T) \times \T; \R^2)$ and $E_{s, \pm} \in C^\infty((0, T) \times \T)$ such that $E_{s,\pm} \ge \gamma$ for some $\gamma > 0$ and
\begin{equation}\label{eq subsoln}
\p_t \boldu_s + \p_x \LB f(\boldu_s) + E_{s,-} \boldb_- + E_{s,+} \boldb_+ \RB = 0.
\end{equation}
\end{definition}

An easy choice for the subsolution is
\begin{equation}\label{const subsoln}
\boldu_s = 0, \quad E_{s,\pm} = \gamma.
\end{equation}

Starting from the above subsolution $(0, \gamma, \gamma)$, we aim to construct a family of approximation $\{ (\boldu_n, E_{n,-}, E_{n,+}) \}$ with $E_{n, \pm} > 0$ such that
\begin{equation}\label{eq iter n}
\p_t \boldu_n + \p_x \LB f(\boldu_n) + E_{n,-} \boldb_- + E_{n,+} \boldb_+ \RB = 0,
\end{equation}
together with further properties that we will discuss in the following. Suppose that $Df(\boldu_n)$ admits two distinct real eigenvalues. 

\subsection{Regularization}
Let $\eta_{\delta}(t,x)$ be a smooth function supported within a space-time cube of sidelength $\delta > 0$. Given a function $f \in L^\infty(\R \times \T)$ we define the regularization of $f$ to be
\begin{equation*}
f^\delta := \eta_\delta \ast f,
\end{equation*}
where the convolution is taken in both space and time. 

Regularizing \eqref{eq iter n} with some scale $\delta_n > 0$ leads to
\begin{equation}\label{eq iter n reg}
\p_t \boldu_n^{\delta_n} + \p_x \LB f(\boldu_n^{\delta_n}) + E_{n,-}^{\delta_n} \boldb_- + E_{n,+}^{\delta_n} \boldb_+ \RB + \p_x \LB {f^{\delta_n}(\boldu_n)} - f(\boldu_n^{\delta_n}) \RB = 0.
\end{equation}

Commutator estimates imply that
\begin{align*}
& \LN {f^{\delta_n}(\boldu_n)} - f(\boldu_n) \RN_{L^\infty} \lesssim \delta_n \| \nabla \boldu_n \|_{L^\infty}, \\
& \LN f(\boldu_n^{\delta_n}) - f(\boldu_n) \RN_{L^\infty} \lesssim \delta_n \| \nabla \boldu_n \|_{L^\infty},
\end{align*}
which yields
\begin{equation}\label{eq comm}
\LN {f^{\delta_n}(\boldu_n)} - f(\boldu_n^{\delta_n}) \RN_{L^\infty} \lesssim {\delta_n} \| \nabla \boldu_n \|_{L^\infty},
\end{equation}
where the constants in these estimates depend only on $f$. For simplicity, we will use $\|\cdot\|$ to indicate the $L^\infty$ norm from now.

\medskip

\noindent{\bf Notation.} For $\delta_n$ sufficiently small we know that $Df(\boldu_n^{\delta_n})$ also has two distinct real eigenvalues. To fix notation, we will denote $\ev_n^\pm := \ev_n^\pm(\boldu_n^{\delta_n})$ to be the two distinct real eigenvalues of $Df(\boldu_n^{\delta_n})$, with the right eigenvectors $\rev_n^\pm := \rev^\pm(\boldu_n^{\delta_n})$. The right eigenvectors of $Df(0)$ are $\rev^\pm$.

\subsection{Localization}\label{subsec loc}
Let $\{ \lambda_n \}^\infty_{n=1}$ be an increasing (super-geometric) sequence with $\lambda_n \to \infty$. For each $n$, let $\{\varphi_{n,j}\}_{j\in \Z}$ be such that $\{\varphi^2_{n,j}\}$ forms a smooth partition of unity for $\R$, that is
\be\label{eq cutoff}
\supp \varphi_{n,j} \subset \LB \frac{j-2/3}{\lambda_n}, \frac{j+2/3}{\lambda_n} \RB, \qquad \sum_{j\in \Z} \varphi^2_{n,j}(\cdot) = 1.
\ee
On the $j$th interval $[(j-2/3)/\lambda_n, (j+2/3)/\lambda_n]$ we define an average of the eigenvalues to be
\begin{equation}\label{eq ave lambda}
\ev_{n,j}^\pm := \frac{j}{\lambda_n}.
\end{equation}
Thus it follows that for $i = \pm$,
\begin{equation}\label{eq est mean osc}
\LV \varphi_{n,j}(\ev_n^i) (\ev_n^i - \ev_{n,j}^i) \RV \le \frac{1}{\lambda_n}.
\end{equation}

If $\boldu_n$ is bounded, say, 
\begin{equation*}
\| \boldu_n \|_{L^\infty} \le M, 
\end{equation*}
then we further have the following derivative estimates for $\varphi_{n,j}$
\begin{equation}\label{est varphi}
\LV \nabla \varphi_{n,j} \LC \ev_n^\pm \RC \RV \lesssim_{M} \lambda_n |\nabla \boldu_n^{\delta_n}|, \qquad \LV \nabla^2 \varphi_{n,j} \LC \ev_n^\pm \RC \RV \lesssim_{M} \lambda_n^2 |\nabla \boldu_n^{\delta_n}|^2 + \lambda_n |\nabla^2 \boldu_n^{\delta_n}|,
\end{equation}
where the constants in the above estimates depend on $M$.

\subsection{Iteration}
Choosing an increasing (super-geometric) sequence $\lambda_n \to \infty$ as above.
Given the $n$-th iteration $\{ (\boldu_n, E_{n,-}, E_{n,+}) \}$, we then choose an appropriate smoothing scale ${\delta_n}$ and amplitude function $\amp(t,x)$, to be determined later, and define
\begin{equation}\label{eq u_n+1}
\boldu_{n+1} = \boldu_n^{\delta_n} + \boldv_{n+1}
\end{equation}
where $\boldv_{n+1}$ has two parts: $\boldv_{n+1} = \boldv^1_{n+1} + \boldv^2_{n+1}$, where $\boldv_{n+1}^1$ is supposed to correct the iteration error at the first order, and $\boldv_{n+1}^2$ is designed to give the second order correction. 

We further make the following decomposition
\[
\boldv^1_{n+1} = \boldv^{1,+}_{n+1} + \boldv^{1,-}_{n+1}, \qquad \boldv^2_{n+1} = \boldv^{2,+,+}_{n+1} + \boldv^{2,+,-}_{n+1} + \boldv^{2,-,+}_{n+1} + \boldv^{2,-,-}_{n+1},
\]
where
\begin{equation}\label{eq v^1}
\boldv^{1,\pm}_{n+1} := \p_x \LCB \sum_j \varphi_{n,j} \LC \ev_n^\pm \RC \frac{\amp}{\lambda_{n+1}} \sin \LB \lambda_{n+1} (x - \ev_{n,j}^\pm t) + P(t) \RB \rev_n^\pm \RCB
\end{equation}
for some phase function $P \in C^2(\R)$ with bounded derivatives, and $\boldv^{2,\pm,\pm}_{n+1}$ will be given later in Section \ref{sec 2nd correct}. Note that the above is a finite sum since $\ev_n^\pm$ is bounded.

From the definition of $\boldv_{n+1}^{1,\pm}$ we have
\begin{align*}
\boldv_{n+1}^{1,\pm} = &  \sum_j \varphi_{n,j}\LC \ev_n^\pm \RC \amp \cos\LB \lambda_{n+1} (x - \ev_{n,j}^\pm t) + P(t) \RB \rev_n^\pm \\
& + \sum_j \p_x \LB \varphi_{n,j} \LC \ev_n^\pm \RC \frac{\amp}{\lambda_{n+1}} \rev_n^\pm \RB \sin \LB \lambda_{n+1} (x - \ev_{n,j}^\pm t) + P(t) \RB.
\end{align*}
Together with \eqref{est varphi} this implies that
\begin{equation}\label{eq est v^1}
\begin{aligned}
& \LV \boldv_{n+1}^{1,\pm} \RV & \lesssim_{M} & \ \LV a_{n+1} \RV \LC 1 + \frac{\lambda_n |\nabla \boldu_n^{\delta_n}|}{\lambda_{n+1}} \RC + \frac{\LV \nabla a_{n+1} \RV}{\lambda_{n+1}}, \\
& \LV \nabla \boldv_{n+1}^{1,\pm} \RV & \lesssim_{M} & \ \LV a_{n+1} \RV \LC \lambda_{n+1} + \lambda_n |\nabla \boldu_n^{\delta_n}| + \frac{\lambda_n^2 |\nabla \boldu_n^{\delta_n}|^2 + \lambda_n {|\nabla^2 \boldu_n^{\delta_n}|}}{\lambda_{n+1}} \RC + \\
& & & \ \LV \nabla a_{n+1} \RV \LC 1 + \frac{\lambda_n |\nabla \boldu_n^{\delta_n}|}{\lambda_{n+1}} \RC + \frac{{\LV \nabla^2 a_{n+1} \RV}}{\lambda_{n+1}},
\end{aligned}
\end{equation}
where 
\[
|a_{n+1}| := \max\{ |a_{n+1}^-|,\  |a_{n+1}^+|\}, \qquad |\nabla a_{n+1}| := \max\{ |\nabla a_{n+1}^-|,\  |\nabla a_{n+1}^+|\}.
\]

We would like to find the equation that $\boldu_{n+1}$ satisfies. Note that
\begin{align}
& \p_t \boldu_{n+1} + \p_x f(\boldu_{n+1}) = \p_t(\boldu_n^{\delta_n} + \boldv_{n+1}) + \p_x f(\boldu_n^{\delta_n} + \boldv_{n+1}) \nonumber \\
= & \ \p_t \boldu_n^{\delta_n} + \p_x f(\boldu_n^{\delta_n}) + \p_t \boldv_{n+1} + \p_x \LB \ev_n \boldv_{n+1} + \LC Df(\boldu_n^{\delta_n}) - \ev_n \id \RC \boldv_{n+1} + \frac{D^2f(\boldu_n^{\delta_n})}{2} : \LC \boldv_{n+1} \otimes \boldv_{n+1} \RC \RB \nonumber \\
& \ + \p_x \underbrace{\LB f(\boldu_n^{\delta_n} + \boldv_{n+1}) - f(\boldu_n^{\delta_n}) - Df(\boldu_n^{\delta_n}) \boldv_{n+1} - \frac{D^2f(\boldu_n^{\delta_n})}{2} : \LC \boldv_{n+1} \otimes \boldv_{n+1} \RC \RB}_{=: \err_1}, \label{full sys}
\end{align}
where we have from the Taylor's theorem that
\begin{equation}\label{eq err1}
\err_1 = O (|\boldv_{n+1}|^3).
\end{equation}

Further plugging in equation \eqref{eq iter n reg} for $\boldu_n^{\delta_n}$ and the decomposition of $\boldv_{n+1}$ we obtain that
\begin{equation}\label{full sys iterate}
\begin{split}
& \ \p_t \boldu_{n+1} + \p_x f(\boldu_{n+1}) \\
= & \ - \p_x\LB \LC E_{n,-}^{\delta_n} \boldb_- + E_{n,+}^{\delta_n} \boldb_+ \RC + \LC f^{\delta_n}(\boldu_n) - f(\boldu_n^{\delta_n}) \RC \RB \\
& \ + \p_t \boldv_{n+1}^1 + \p_x \LB  Df(\boldu_n^{\delta_n}) \boldv_{n+1}^1 + \frac{D^2f(\boldu_n^{\delta_n})}{2} : \LC \boldv_{n+1}^1 \otimes \boldv_{n+1}^1 \RC \RB \\
& \ + \p_t \boldv_{n+1}^2 + \p_x \LB  Df(\boldu_n^{\delta_n}) \boldv_{n+1}^2 \RB + \p_x (\err_1 + \err_2),
\end{split}
\end{equation}
where 
{\small
\begin{equation}\label{eq err2}
\err_2 := \frac{D^2f(\boldu_n^{\delta_n})}{2} : \LB \LC \boldv_{n+1} \otimes \boldv_{n+1} \RC - \LC \boldv_{n+1}^1 \otimes \boldv_{n+1}^1 \RC \RB = O \LC |\boldv_{n+1}^1| |\boldv_{n+1}^2| + |\boldv_{n+1}^2|^2 \RC.
\end{equation}}

\subsection{First order correction}
The first part of the $(n+1)$-st oscillation, $\boldv^{1,\pm}_{n+1}$, is supposed to decrease the error at the linear level. We will leave most of the technical estimates in Appendix \ref{sec appendix correct}. One can check that
\[
\p_t \boldv^{1,\pm}_{n+1} + \p_x \LC \ev_n^\pm \boldv^{1,\pm}_{n+1} \RC = \p_x \rem^{(1),\pm}_{n+1}
\]
where $\rem^{(1),\pm}_{n+1}$ is given in \eqref{eq remainder 1}, with the estimates in \eqref{eq est R_1}.

Moreover, let
\begin{equation*}
\rem^{(2),\pm}_{n+1} := \LB Df(\boldu_n^{\delta_n}) - \ev_n^\pm \id \RB \boldv_{n+1}^{1, \pm}.
\end{equation*}
Then using the fact that $\LB Df(\boldu_n^{\delta_n}) - \ev_n^\pm \id \RB \rev_n^\pm = 0$, an improved estimate can be obtained as in \eqref{eq est R_2}.

Now for the quadratic terms we have for $k, l \in \{+, -\}$,
\begin{equation*}
\begin{split}
& D^2f(\boldu_n^{\delta_n}) : \LC \boldv_{n+1}^{1, k} \otimes \boldv_{n+1}^{1, l} \RC \\
& \ = \sum_{i,j} \varphi_{n,i} \LC \ev_n^k \RC \varphi_{n,j} \LC \ev_n^l \RC \cos \LB \lambda_{n+1} (x - \ev_{n,i}^k t) + P(t) \RB \cos \LB \lambda_{n+1} (x - \ev_{n,j}^l t) + P(t) \RB \cdot \\
& \qquad \quad  \LC a_{n+1}^k \cdot a_{n+1}^l \RC \LB D^2f(0) : \LC \rev^k \otimes \bold\rev^l \RC \RB  + \rem^{(3),k,l}_{n+1},
\end{split}
\end{equation*}
where the remainder $\rem^{(3),k,l}_{n+1}$ and the corresponding estimates are given in \eqref{eq def R_3} and \eqref{eq est R_3} respectively.

Putting together we find that $\boldv_{n+1}^{1}$ satisfies
{\small
\begin{align}\label{1st correct}
&  \p_t \boldv_{n+1}^{1} + \p_x \LB Df(\boldu_n^{\delta_n}) \boldv_{n+1}^{1} \RB + \frac{D^2f(\boldu_n^{\delta_n})}{2} : \LC \boldv_{n+1}^{1} \otimes \boldv_{n+1}^{1} \RC \nonumber \\
= & \sum_{k = \pm} \Big\{ \p_t \boldv_{n+1}^{1, k} + \p_x \LB Df(\boldu_n^{\delta_n}) \boldv_{n+1}^{1, k} \RB \Big\} + \frac{D^2f(\boldu_n^{\delta_n})}{2} : \LC \sum_{k=\pm} \boldv_{n+1}^{1, k} \otimes \sum_{l = \pm} \boldv_{n+1}^{1, l} \RC \\
= & \ \frac12 \p_x \sum_{k,l = \pm} \LCB \sum_{i,j} \varphi_{n,i} ( \ev_n^k ) \varphi_{n,j} ( \ev_n^l ) \cos \LB \lambda_{n+1} (x - \ev_{n,i}^k t) + P(t) \RB \cos \LB \lambda_{n+1} (x - \ev_{n,j}^l t) + P(t) \RB \cdot \right. \nonumber\\
& \ \  \quad  \LC a_{n+1}^k \cdot a_{n+1}^l \RC \LB D^2f(0) : \LC \rev^k \otimes \rev^l \RC \RB \Big\}  + \p_x \LB \sum_{k = \pm} \LC \rem^{(1),k}_{n+1} + \rem^{(2),k}_{n+1} \RC + \sum_{k,l = \pm} \rem^{(3),k,l}_{n+1} \RB \nonumber\\
=: & \ \frac12 \p_x \sum_{k,l = \pm} \LCB \LC a_{n+1}^k \cdot a_{n+1}^l \RC Q^{k,l}_{n+1} \LB D^2f(0) : \LC \rev^k \otimes \rev^l \RC \RB   \RCB + \p_x \LB \sum_{k = \pm} \LC \rem^{(1),k}_{n+1} + \rem^{(2),k}_{n+1} \RC + \sum_{k,l = \pm} \rem^{(3),k,l}_{n+1} \RB. \nonumber
\end{align}}
This corresponds to the third line of \eqref{full sys iterate}.

\subsection{Second order correction}\label{sec 2nd correct}
From above we find that with a controllable error, the first part of the oscillation $\boldv_{n+1}^{1, \pm}$ ``corrects'' the equation up to a quadratic error 
\begin{align*}
& \p_x \LCB \frac12 \sum_{k,l = \pm} \LC a_{n+1}^k \cdot a_{n+1}^l \RC Q^{k,l}_{n+1} \LB D^2f(0) : \LC \rev^k \otimes \rev^l \RC \RB   \RCB \\
= & \ \frac12 \p_x \LCB  \sum_{k = \pm} \LC a_{n+1}^k \RC^2 Q^{k,k}_{n+1}  \boldb_k + \Big[ \LC a_{n+1}^+ \cdot a_{n+1}^- \RC Q^{+,-}_{n+1} + \LC a_{n+1}^- \cdot a_{n+1}^+ \RC  Q^{-,+}_{n+1}  \Big] \boldsymbol{d} \RCB,
\end{align*}
where $\boldb_\pm$ and $\boldsymbol{d}$ are defined in \eqref{eq def b and d}. Note that this error term involves the interaction between two cosine waves. By symmetry we know that $Q^{+,-}_{n+1} = Q^{-,+}_{n+1}$.

Explicit calculation leads to
\begin{align*}
2Q^{k,k}_{n+1} = & \ \sum_j \LB \varphi_{n,j} \LC \ev_n^k \RC \RB^2 \LB 1 + \cos \LC 2\lambda_{n+1} (x - \ev_{n,j}^k t) + 2P(t) \RC \RB + \\
& \ \sum_{|i-j|=1} \varphi_{n,i} \LC \ev_n^k \RC \varphi_{n,j} \LC \ev_n^k \RC \cos \LC \lambda_{n+1} (2x - (\ev_{n,i}^k + \ev_{n,j}^k) t) + 2P(t) \RC + \\
& \ \sum_{|i-j| = 1}  \varphi_{n,i} \LC \ev_n^k \RC \varphi_{n,j} \LC \ev_n^k \RC \cos \LC \lambda_{n+1} (\ev_{n,j}^k - \ev_{n,i}^k) t \RC, \quad \text{for } k \in \{+, -\}.
\end{align*}
From \eqref{eq ave lambda} we know that $\ev_{n,j}^+ = \ev_{n,j}^-$, and hence
\begin{align*}
2Q^{+,-}_{n+1} = & \ \sum_j  \varphi_{n,j} \LC \ev_n^+ \RC \varphi_{n,j} \LC \ev_n^- \RC  \LB 1 + \cos \LC 2\lambda_{n+1} (x - \ev_{n,j}^+ t) + 2P(t) \RC \RB + \\
& \ \sum_{i\ne j} \varphi_{n,i} \LC \ev_n^+ \RC \varphi_{n,j} \LC \ev_n^- \RC \cos \LC \lambda_{n+1} (2x - (\ev_{n,i}^+ + \ev_{n,j}^-) t) + 2P(t) \RC + \\
& \ \sum_{i \ne j}  \varphi_{n,i} \LC \ev_n^+ \RC \varphi_{n,j} \LC \ev_n^- \RC \cos \LC \lambda_{n+1} (\ev_{n,j}^+ - \ev_{n,i}^-) t \RC.
\end{align*}

Consider the first term on the right-hand side of the above. Roughly, we expect $\boldu_n$ to be very small, and hence $\ev_n^\pm$ is close to $\ev^\pm(0)$, say
\[
\LV \ev^\pm_n - \ev^\pm(0) \RV \le \frac{\LV \ev^{\pm}(0) \RV}{2}.
\]
This implies that $\ev_n^\pm$ remain separate due to strict hyperbolicity at 0. In particular by taking
\[
\lambda_0 > \frac{4}{|\ev^+(0) - \ev^-(0)|}
\]
we have that $\varphi_{n,j} \LC \ev_n^+ \RC \varphi_{n,j} \LC \ev_n^- \RC = 0$ for all $j$. 

The goal is to correct the above quadratic error using the second part of the oscillation: $\boldv_{n+1}^{2, \pm,\pm}$. To balance those oscillating terms it is natural to consider $\boldv_{n+1}^{2}$ of the form
\[
\boldv_{n+1}^{2} = \boldv^{2,+,+}_{n+1} + \boldv^{2,+,-}_{n+1} + \boldv^{2,-,+}_{n+1} + \boldv^{2,-,-}_{n+1},
\]
where
\begin{subequations}\label{eq v^2}
{\small
\begin{align}\label{eq v^2 ++}
\boldv_{n+1}^{2, k,k} = & \ - \p_x \LC \sum_j \frac{\LC a_{n+1}^k \RC^2}{8\lambda_{n+1}}  \LB \varphi_{n,j} \LC \ev_n^k \RC \RB^2 \sin \LB 2\lambda_{n+1} (x - \ev_{n,j}^k t) + 2P(t) \RB \boldB_k \RC - \\
& \ \p_x \LC \sum_{|i-j| = 1} \frac{\LC a_{n+1}^k \RC^2}{8\lambda_{n+1}}  \varphi_{n,i} \LC \ev_n^k \RC \varphi_{n,j} \LC \ev_n^k \RC \sin \LB \lambda_{n+1} (2x - (\ev_{n,i}^k + \ev_{n,j}^k) t) + 2P(t) \RB \boldB_k \RC - \nonumber \\
& \ \p_x \LC \sum_{|i-j| = 1} \frac{\LC a_{n+1}^k \RC^2}{4 \lambda_{n+1} (\ev_{n,j}^k - \ev_{n,i}^k)}   \varphi_{n,i} \LC \ev_n^k \RC \varphi_{n,j} \LC \ev_n^k \RC \sin \LC \lambda_{n+1}  (\ev_{n,j}^k - \ev_{n,i}^k) t \RC \boldb_k \RC, \nonumber
\end{align}}
for $k \in \{+, -\}$, and $\boldB_\pm$ are defined in \eqref{eq exist B},
and
{\small
\begin{align}\label{eq v^2 +-}
\boldv_{n+1}^{2, +,-} = & \ \boldv_{n+1}^{2, -,+} \\
= & \ - \p_x \LC \sum_{i \ne j} \frac{ a_{n+1}^+ \cdot a_{n+1}^- }{8\lambda_{n+1}}  \varphi_{n,i} \LC \ev_n^+ \RC \varphi_{n,j} \LC \ev_n^- \RC \sin \LB \lambda_{n+1} (2x - (\ev_{n,i}^+ + \ev_{n,j}^-) t) + 2P(t) \RB \boldD \RC - \nonumber \\
& \ \p_x \LC \sum_{i \ne j} \frac{ a_{n+1}^+ \cdot a_{n+1}^- }{4 \lambda_{n+1} (\ev_{n,j}^+ - \ev_{n,i}^-)}   \varphi_{n,i} \LC \ev_n^+ \RC \varphi_{n,j} \LC \ev_n^- \RC \sin \LC \lambda_{n+1}  (\ev_{n,j}^+ - \ev_{n,i}^-) t \RC \boldd \RC, \nonumber
\end{align}}
where $\boldD$ is such that
\begin{equation*}
\LC Df(0) - \frac{\ev^+(0) + \ev^-(0)}{2} \id \RC \boldD = \boldd.
\end{equation*}
Note that the existence of $\boldD$ is the consequence of strict hyperbolicity at 0.
\end{subequations}

Similar to \eqref{eq est v^1}, we have the following estimate for $\boldv_{n+1}^{2,\pm,\pm}$: for $k, l \in \{+, -\}$,
{\small
\begin{align}\label{eq est v^2}
& \LV \boldv_{n+1}^{2,k,l} \RV & \lesssim_{M} & \ a_{n+1}^2 \LC 1 + \frac{\lambda_n^2 |\nabla \boldu_n^{\delta_n}|}{\lambda_{n+1}} \RC + \frac{\lambda_n |a_{n+1}| |\nabla a_{n+1}|}{\lambda_{n+1}}, \nonumber \\
& \LV \nabla \boldv_{n+1}^{2,k,l} \RV & \lesssim_{M} & \ a_{n+1}^2 \LC \lambda_{n+1} + \lambda_n^2 |\nabla \boldu_n^{\delta_n}| + \frac{\lambda_n^3 |\nabla \boldu_n^{\delta_n}|^2 + \lambda_n^2 {|\nabla^2 \boldu_n^{\delta_n}|}}{\lambda_{n+1}} \RC \\
& & & + |a_{n+1}| |\nabla a_{n+1}| \LC \lambda_n + \frac{\lambda_n^2 |\nabla \boldu_n^{\delta_n}|}{\lambda_{n+1}} \RC + \frac{\lambda_n}{\lambda_{n+1}} \LC |\nabla a_{n+1}|^2 + |a_{n+1}| {|\nabla^2 a_{n+1}|} \RC. \nonumber
\end{align}}

This way we know that we only need to take into account of the contribution from $\boldv_{n+1}^{2}$ to the system \eqref{full sys} from the linear terms (corresponding to the fourth line of \eqref{full sys iterate}) of the form
\[
\p_t \boldv^{2,k,k}_{n+1} + \p_x \LC \ev_n^k \boldv^{2,k,k}_{n+1} \RC + \p_x \LB \LC Df(\boldu_n^{\delta_n}) - \ev_n^k \id \RC \boldv_{n+1}^{2,k,k} \RB, \quad k \in \{+, -\},
\] 
and
\[
\p_t \boldv^{2,k,l}_{n+1} + \p_x \LC \frac{\ev_n^k + \ev_n^l}{2} \boldv^{2,k,l}_{n+1} \RC + \p_x \LB \LC Df(\boldu_n^{\delta_n}) - \frac{\ev_n^+ + \ev_n^-}{2} \id \RC \boldv_{n+1}^{2,k,l} \RB, \quad k \ne l \in \{+, -\}.
\]
where the last terms in the above can be replaced by 
\[
\p_x \LB \LC Df(\boldu_n^{\delta_n}) - \ev_n^k \id \RC \boldv_{n+1}^{2,k,k} \RB = \p_x \LB \LC Df(0) - \ev^k(0) \id \RC \boldv_{n+1}^{2,k,k} + \err^{k,k}_3 \RB, \quad k \in \{+, -\},
\]
with
\begin{equation}\label{eq err2 kk}
\err^{k,k}_3 := \LB \LC Df(\boldu_n^{\delta_n}) - \ev_n^k \id \RC - \LC Df(0) - \ev^k(0) \id \RC \RB \boldv_{n+1}^{2,k,k} =  O \LC \LV \boldu_n^{\delta_n} \RV \LV \boldv_{n+1}^{2,k,k} \RV \RC,
\end{equation}
and for $k \ne l \in \{+, -\}$,
\[
\p_x \LB \LC Df(\boldu_n^{\delta_n}) - \frac{\ev_n^+ + \ev_n^-}{2} \id \RC \boldv_{n+1}^{2,k,l} \RB = \p_x \LB \LC Df(0) - \frac{\ev^+(0) + \ev^-(0)}{2} \id \RC \boldv_{n+1}^{2,k,l} + \err^{k,l}_3 \RB,
\]
where
\begin{equation}\label{eq err2 kl}
\begin{split}
\err^{k,l}_3 := & \ \LB \LC Df(\boldu_n^{\delta_n}) - \frac{\ev_n^+ + \ev_n^-}{2} \id \RC - \LC Df(0) - \frac{\ev^+(0) + \ev^-(0)}{2} \id \RC \RB \boldv_{n+1}^{2,k,l} \\
= & \ O \LC \LV \boldu_n^{\delta_n} \RV \LV \boldv_{n+1}^{2,k,l} \RV \RC.
\end{split}
\end{equation}

Following the same argument as before in obtaining \eqref{eq est R_1} we have for $k \in \{+, -\}$,
\begin{align*}
\p_t \boldv^{2,k,k}_{n+1} + \p_x \LC \ev_n^k \boldv^{2,k,k}_{n+1} \RC = & \  - \p_x \LB \frac14 \LC a_{n+1}^k \RC^2 \sum_{|i-j| = 1} \varphi_{n,i} \LC \ev_n^k \RC \varphi_{n,j} \LC \ev_n^k \RC \cos \LC \lambda_{n+1} (\ev_{n,j}^k - \ev_{n,i}^k) t \RC \boldb_k  \RB \\
&\  - \p_x \rem^{(4),k,k}_{n+1}, 
\end{align*}
where $\rem^{(4),k,k}_{n+1}$ is given in \eqref{eq def R_4 1}.

\medskip

Similarly, we obtain that
{\small
\begin{align*}
\p_t \boldv^{2,+,-}_{n+1} + \p_x \LC \frac{\ev_n^+ + \ev_n^-}{2} \boldv^{2,+,-}_{n+1} \RC = & \ \p_t \boldv^{2,-,+}_{n+1} + \p_x \LC \frac{\ev_n^- + \ev_n^+}{2} \boldv^{2,-,+}_{n+1} \RC \\
= & \  - \p_x \LB \frac{a_{n+1}^+ \cdot a_{n+1}^-}{4}  \sum_{i\ne j} \varphi_{n,i} \LC \ev_n^+ \RC \varphi_{n,j} \LC \ev_n^- \RC \cos \LC \lambda_{n+1} (\ev_{n,j}^+ - \ev_{n,i}^-) t \RC \boldd  \RB \\
&\  - \p_x \rem^{(4),+,-}_{n+1} 
\end{align*}}

\noindent where $\rem^{(4),+,-}_{n+1}$ is defined in \eqref{eq def R_4 2}. The estimates for $\rem^{(4),k,l}_{n+1}$, $k,l \in \{+,-\}$, are provided in \eqref{eq est R_4}.

Finally, recalling the definition of $\tilde{\boldb}_\pm$ from \eqref{eq exist B}, we have for $k \in \{+, -\}$ that
{\small
\begin{align*}
& \LC Df(0) - \ev^k(0) \id \RC \boldv_{n+1}^{2,k,k} + \frac12 \LC a_{n+1}^k \RC^2   Q^{k,k}_{n+1} \boldb_k \\
=: \ & \frac14 \LC a_{n+1}^k \RC^2  \LB \sum_j \LB \varphi_{n,j} \LC \ev_n^k \RC \RB^2 + \sum_{|i-j|=1} \varphi_{n,i} \LC \ev_n^k \RC \varphi_{n,j} \LC \ev_n^k \RC \cos \LC \lambda_{n+1} (\ev_{n,j}^k - \ev_{n,i}^k) t \RC \RB \boldb_k \\
& + \frac14 \LC a_{n+1}^k \RC^2 \sum_{|i-j|=1} \varphi_{n,i} \LC \ev_n^k \RC \varphi_{n,j} \LC \ev_n^k \RC \cos \LB \lambda_{n+1} (2x - (\ev_{n,i}^k + \ev_{n,j}^k) t) + 2P(t) \RB (\boldb_k -\tilde{\boldb}_k) + \rem^{(5),k,k}_{n+1},
\end{align*}}
and
{\small
\begin{align*}
& \LC Df(0) - \frac{\ev^+(0) + \ev^-(0)}{2} \id \RC \boldv_{n+1}^{2,+,-} + \frac12 (a_{n+1}^+ \cdot a_{n+1}^-) Q_{n+1}^{+,-} \boldd \\
=: \ & \frac{a_{n+1}^+ \cdot a_{n+1}^-}{4} \sum_{i\ne j} \varphi_{n,i} \LC \ev_n^+ \RC \varphi_{n,j} \LC \ev_n^- \RC \cos \LC \lambda_{n+1} (\ev_{n,j}^+ - \ev_{n,i}^-) t \RC \boldd + \rem^{(5),+,-}_{n+1}.
\end{align*}}

\noindent where $\rem^{(5),k,l}_{n+1}$ are given in \eqref{eq def R_5 1} and \eqref{eq def R_5 2}. Similar calculation applies to $\boldv_{n+1}^{2,-,+}$. 

The estimates for $\rem^{(5),k,l}_{n+1}$ can be found in \eqref{eq est R_5}. 

\bigskip

Putting together and using \eqref{eq cutoff} and \eqref{eq decomp b}--\eqref{eq exist B} yields
\begin{align}\label{2nd correct} 
& \p_t \boldv_{n+1}^2 + \p_x \LB Df(\boldu_{n+1}^{\delta_n}) \boldv_{n+1}^2 \RB = \sum_{k, l = \pm} \LCB \p_t \boldv_{n+1}^{2,k,l} + \p_x \LB Df(\boldu_{n+1}^{\delta_n}) \boldv_{n+1}^{2,k,l} \RB \RCB \nonumber\\
=\ & - \p_x \LCB \frac12 \sum_{k,l = \pm} \LC a_{n+1}^k \cdot a_{n+1}^l \RC Q^{k,l}_{n+1} \LB D^2f(0) : \LC \rev^k \otimes \rev^l \RC \RB   \RCB \nonumber \\
& + \p_x \LCB \frac14 \sum_{k = \pm} \LC a_{n+1}^k \RC^2  \sum_j \LB \varphi_{n,j} \LC \ev_n^k \RC \RB^2  \boldb_k \RCB \\
& + \p_x \LCB \frac14 \sum_{k = \pm} \LC a_{n+1}^k \RC^2 \sum_{|i-j|=1} \varphi_{n,i} \LC \ev_n^k \RC \varphi_{n,j} \LC \ev_n^k \RC \cos \LB \lambda_{n+1} (2x - (\ev_{n,i}^k + \ev_{n,j}^k) t) + 2P(t) \RB (\boldb_k -\tilde{\boldb}_k) \RCB  \nonumber \\
& + \p_x \sum_{k,l = \pm} \LB \rem^{(4),k,l}_{n+1} + \rem^{(5),k,l}_{n+1} + \err_3^{k,l} \RB  \nonumber \\
=\ & - \p_x \LCB \frac12 \sum_{k,l = \pm} \LC a_{n+1}^k \cdot a_{n+1}^l \RC Q^{k,l}_{n+1} \LB D^2f(0) : \LC \rev^k \otimes \rev^l \RC \RB   \RCB + \p_x \LCB \frac14 \sum_{k = \pm} \LC a_{n+1}^k \RC^2  (1 + s_k)  \boldb_k \RCB   \nonumber \\
& + \p_x \sum_{k,l = \pm} \LB \rem^{(4),k,l}_{n+1} + \rem^{(5),k,l}_{n+1} + \err_3^{k,l} \RB, \nonumber
\end{align}
where 
\begin{align*}
s_\pm := & \sum_{|i-j|=1} \alpha^{\pm} \varphi_{n,i} \LC \ev_n^\pm \RC \varphi_{n,j} \LC \ev_n^\pm \RC \cos \LB \lambda_{n+1} (2x - (\ev_{n,i}^\pm + \ev_{n,j}^\pm) t) + 2P(t) \RB + \\
& \sum_{|i-j|=1} \beta^{\mp} \varphi_{n,i} \LC \ev_n^\mp \RC \varphi_{n,j} \LC \ev_n^\mp \RC \cos \LB \lambda_{n+1} (2x - (\ev_{n,i}^\mp + \ev_{n,j}^\mp) t) + 2P(t) \RB.
\end{align*}
From Lemma \ref{lemm structure} we have
\be\label{eq est s}
|s_\pm| \le \frac{\ep}{1 - \ep}, \qquad |\nabla s_\pm| \lesssim_M \ep \LC \lambda_{n+1} + \lambda_n |\nabla \boldu_n^{\delta_n}| \RC.
\ee

\subsection{System at $(n+1)$st iteration}
With all of the above effort, we finally arrive at the system satisfied by $\boldu_{n+1}$:
\begin{align}\label{eqn syst at n+1}
& \ \  \p_t \boldu_{n+1} + \p_x f(\boldu_{n+1}) \nonumber \\
= & \ - \p_x\LB \LC E_{n,-}^{\delta_n} \boldb_- + E_{n,+}^{\delta_n} \boldb_+ \RC + \LC f^{\delta_n}(\boldu_n) - f(\boldu_n^{\delta_n}) \RC \RB \nonumber \\
& \ + \p_x \LCB \frac14 \sum_{k = \pm} \LC a_{n+1}^k \RC^2 (1 + s_k)  \boldb_k \RCB \\
& \ + \p_x \underbrace{\LC \sum^2_{i=1} \sum_{k = \pm} \rem^{(i),k}_{n+1} + \sum^5_{i = 3} \sum_{k,l = \pm} \rem^{(i),k,l}_{n+1} + \err_1 + \err_2 + \sum_{k,l = \pm} \err^{k,l}_3  \RC}_{=: -\boldW_{n+1}}, \nonumber
\end{align}
which is equivalent to
{\small
\begin{equation}\label{eqn syst at n+1}
\begin{split}
\p_t \boldu_{n+1} + \p_x \LB f(\boldu_{n+1}) + \sum_{k = \pm} \LC E_{n,k}^{\delta_n}  - \frac{1+s_k}{4} \LC a_{n+1}^k \RC^2  \RC \boldb_k + \boldW_{n+1} + {f^{\delta_n}(\boldu_n)} - f(\boldu_n^{\delta_n}) \RB = 0.
\end{split}
\end{equation}}

This way we can complete the $(n+1)$st iteration $(\boldu_{n+1}, E_{n+1, \pm})$ by setting
\begin{equation}\label{eq E_n+1}
E_{n+1, \pm} = E_{n,\pm}^{\delta_n}  - \frac{1+s_\pm}{4} \LC \amp \RC^2 + w_{n+1,\pm},
\end{equation}
where $w_{n+1,\pm}$ are obtained through Cramer's rule
\begin{equation}\label{eq def w}
\begin{split}
& w_{n+1,+} = \frac{\det \LC \boldb_-, \boldW_{n+1} + f^{\delta_n}(\boldu_n) - f(\boldu_n^{\delta_n}) \RC}{\det (\boldb_-, \boldb_+)}, \\    
& w_{n+1,-} = \frac{\det \LC \boldW_{n+1} + f^{\delta_n}(\boldu_n) - f(\boldu_n^{\delta_n}), \boldb_+ \RC}{\det (\boldb_-, \boldb_+)}.  
\end{split}
\end{equation}
Keep in mind that at this stage the choice for $\amp$ is completely open.

\subsection{Estimate on $\boldW_{n+1}$}\label{subsec W}
To obtain the estimate for $W_{n+1}$, we further deduce from \eqref{full sys}, \eqref{eq err2}, \eqref{eq err2 kk}, and \eqref{eq err2 kl} that
\begin{equation}\label{eq est err}
\err_1 \lesssim_{M} |\boldv_{n+1}|^3, \qquad \err_2 \lesssim_{M} |\boldv_{n+1}| |\boldv_{n+1}^2|, \qquad  \err_3^{k,l} \lesssim_{M} \LV \boldu_n^{\delta_n} \RV | \boldv_{n+1}^{2,k,l} |,
\end{equation}
and 
\begin{equation}\label{eq est grad err}
\begin{split}
\LV \nabla \err_1 \RV & \lesssim_{M} (1 + |\boldv_{n+1}| + |\boldv_{n+1}|^2) (|\nabla \boldu_n^{\delta_n}| + |\nabla \boldv_{n+1}|),\\ 
\LV \nabla \err_2 \RV & \lesssim_{M} \LV \nabla \boldu_n^{\delta_n} \RV |\boldv_{n+1}| |\boldv_{n+1}^2| + |\nabla \boldv_{n+1}^1| |\boldv_{n+1}^2| + |\boldv_{n+1}| |\nabla \boldv_{n+1}^2|,  \\
\LV \nabla \err_3^{k,l} \RV & \lesssim_{M} \LV \nabla \boldu_n^{\delta_n} \RV | \boldv_{n+1}^{2,k,l} | + \LV \nabla \boldv_{n+1}^{2,k,l} \RV.
\end{split}
\end{equation}

Putting together, the estimates on $\boldW_{n+1}$ read
\begin{align*}
|\boldW_{n+1}| \lesssim_{M} & \LV a_{n+1} \RV \LC \frac{ \lambda_n |\nabla \boldu_n^{\delta_n}|}{\lambda_{n+1}} + \frac{1}{\lambda_n} \RC + \frac{\LV \nabla a_{n+1} \RV}{\lambda_{n+1}} \\
& + a_{n+1}^2 \LB |\boldu_n^{\delta_n}|  \LC 1 + \frac{\lambda_n |\nabla \boldu_n^{\delta_n}|}{\lambda_{n+1}} \RC^2 + \LC \frac{\lambda_n |\nabla \boldu_n^{\delta_n}|}{\lambda_{n+1}} \RC^2 + |\boldu_n^{\delta_n}|^2 \RB + \frac{\LV \nabla a_{n+1} \RV^2 |\boldu_n^{\delta_n}|}{\lambda_{n+1}^2} \\
& + a_{n+1}^2 \LC \frac{ \lambda_n^2 |\nabla \boldu_n^{\delta_n}|}{\lambda_{n+1}} + \frac{1}{\lambda_{n}} \RC + \frac{\lambda_n |a_{n+1}| |\nabla a_{n+1}|}{\lambda_{n+1}} \\
& + \LB \LV a_{n+1} \RV \LC 1 + \frac{\lambda_n |\nabla \boldu_n^{\delta_n}|}{\lambda_{n+1}} \RC + \frac{\LV \nabla a_{n+1} \RV}{\lambda_{n+1}} + a_{n+1}^2 \LC 1 + \frac{\lambda_n^2 |\nabla \boldu_n^{\delta_n}|}{\lambda_{n+1}} \RC + \frac{\lambda_n |\boldu_n^{\delta_n}| |\nabla \boldu_n^{\delta_n}|}{\lambda_{n+1}} \RB^3 \\
& + |\boldu_n^{\delta_n}| \LB a_{n+1}^2 \LC 1 + \frac{\lambda_n^2 |\nabla \boldu_n^{\delta_n}|}{\lambda_{n+1}} \RC + \frac{\lambda_n |\boldu_n^{\delta_n}| |\nabla \boldu_n^{\delta_n}|}{\lambda_{n+1}} \RB^2, \\
|\nabla \boldW_{n+1}| \lesssim_{M} & \frac{\LV a_{n+1} \RV  \lambda_{n+1}}{\lambda_n} (1 + |a_{n+1}|) + |a_{n+1}| (1 + |a_{n+1}| \lambda_n) \LC \lambda_n |\nabla \boldu_n^{\delta_n}| +  \frac{\lambda_n^2 |\nabla \boldu_n^{\delta_n}|^2 + \lambda_n {|\nabla^2 \boldu_n^{\delta_n}|}}{\lambda_{n+1}} \RC + \\
& (1 + |a_{n+1}| \lambda_n) \LB \LV \nabla a_{n+1} \RV \LC \frac{\lambda_n |\nabla \boldu_n^{\delta_n}|}{\lambda_{n+1}} + \frac{1}{\lambda_n} \RC + \frac{{\LV \nabla^2 a_{n+1} \RV}}{\lambda_{n+1}} \RB + \frac{\lambda_n}{\lambda_{n+1}} |\nabla a_{n+1}|^2 +  \\
& |\nabla \boldu_n^{\delta_n}| \LV \boldv_{n+1}^{1} \RV^2 + |\boldu_n^{\delta_n}| \LV \boldv_{n+1}^{1} \RV {\LV \nabla \boldv_{n+1}^{1} \RV} + a_{n+1}^2 \LC |\boldu_n^{\delta_n}|^2 \RC \LC \lambda_n |\nabla \boldu_n^{\delta_n}| + \lambda_{n+1} \RC + \\
& a_{n+1}^2 |\boldu_n^{\delta_n}| |\nabla \boldu_n^{\delta_n}| + \frac{a_{n+1}^2 \lambda_n^2 |\nabla \boldu_n^{\delta_n}|}{\lambda_{n+1}} \LC |\nabla \boldu_n^{\delta_n}| + \frac{\lambda_n |\nabla \boldu_n^{\delta_n}|^2 + {|\nabla^2 \boldu_n^{\delta_n}|}}{\lambda_{n+1}} \RC + \\
& |a_{n+1}| |\nabla a_{n+1}| \LC \ep + |\boldu_n^{\delta_n}|^2 + \frac{\lambda_n^2 |\nabla \boldu_n^{\delta_n}|^2}{\lambda_{n+1}^2} \RC + \LV \nabla \err_1 \RV +  \LV \nabla \err_2 \RV + \sum_{k,l = \pm} \LV \nabla \err_3^{k,l} \RV.
\end{align*}

\section{Choice for $\amp$}\label{sec amp}
The goal is to choose some appropriate $\amp$ such that $E_{n,\pm} \to 0$ as $n\to \infty$. We pick two parameters $0 < \beta < \gamma < 1$ to be determined later, and define a smooth function $\phi_{\beta, \gamma} : [0, \infty) \to [0,1]$ such that 
\begin{equation*}
\phi_{\beta, \gamma}(s) = \left\{\begin{aligned}
 & 0, & \quad & \text{ when } 0 \le s \le \beta, \\
 & 1, & & \text{ when } \gamma \le s,
\end{aligned}\right. 
\end{equation*}
and $\phi_{\beta, \gamma}$ is nondecreasing; see Figure \ref{fig phi}.
\begin{figure}[h]
  \includegraphics[page=3,scale=0.85]{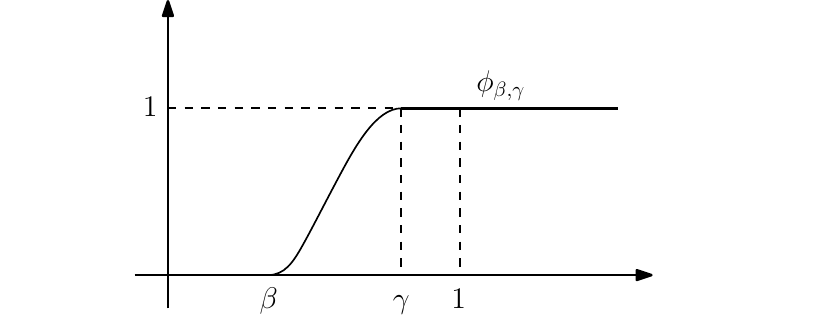}
  \vspace{-2ex}
  \caption{The graph of $\phi_{\beta, \gamma}$}
  \label{fig phi}
\end{figure}

Looking to obtain a bound on $E_{n,\pm}$ of the kind that 
\[
0 \le E_{n,\pm} \le F_n \qquad \text{with} \quad 0 < F_n \searrow 0 \text{ as } n\to \infty,
\]
we will choose $\amp \ge 0$ to be such that
\begin{equation}\label{eq choice for a}
\LC \amp \RC^2 = 2 \phi_{\beta, \gamma}^2 \LC \frac{E_{n,\pm}^{\delta_n}}{F_n} \RC E_{n, \pm}^{\delta_n},
\end{equation}
this leads to
\begin{align*}
\frac{1}{\sqrt2}\nabla \amp = & \ \LB \phi'_{\beta, \gamma} \LC {E_{n,\pm}^{\delta_n}}/{F_n} \RC \frac{\sqrt{E_{n,\pm}^{\delta_n}}}{F_n} \nabla E_{n, \pm}^{\delta_n} +   \phi_{\beta, \gamma} \LC {E_{n,\pm}^{\delta_n}}/{F_n} \RC \frac{\nabla E_{n, \pm}^{\delta_n}}{2 \sqrt{E_{n, \pm}^{\delta_n}}} \RB, \\
\frac{1}{\sqrt2} \nabla^2 \amp = & \ \LB \phi''_{\beta, \gamma}  \frac{\sqrt{E_{n,\pm}^{\delta_n}}}{F_n^2} + \frac{ \phi'_{\beta, \gamma} }{2 F_n \sqrt{E_{n, \pm}^{\delta_n}}} - \frac{\phi_{\beta, \gamma} }{\LC E_{n, \pm}^{\delta_n}\RC^{3/2} } \RB \nabla E_{n, \pm}^{\delta_n} \otimes \nabla E_{n, \pm}^{\delta_n}  \\
& \quad + \LC \phi'_{\beta, \gamma} \frac{\sqrt{E_{n,\pm}^{\delta_n}}}{F_n} + \frac{\phi_{\beta, \gamma}}{2 \sqrt{E_{n, \pm}^{\delta_n}}} \RC \nabla^2 E_{n, \pm}^{\delta_n}.
\end{align*}
Note that $\phi_{\beta, \gamma}(s) \ne 0$ only for $s \ge \beta$, and $\phi'_{\beta, \gamma}(s) \ne 0$ for $\beta \le s \le \gamma$. Thus
\begin{equation}\label{eq est amp}
\begin{aligned}
& |\amp| \le \sqrt{2 F_n},  \qquad |\nabla \amp | \lesssim_{\beta,\gamma} \frac{\LV \nabla E_{n,\pm}^{\delta_n} \RV}{\sqrt{F_n}}, \\
& |\nabla^2 \amp | \lesssim_{\beta, \gamma} \frac{\LV \nabla^2 E_{n,\pm}^{\delta_n} \RV}{\sqrt{F_n}} + \frac{\LV \nabla E_{n,\pm}^{\delta_n} \RV^2}{F_n^{3/2}}.
\end{aligned}
\end{equation}

\section{Induction argument and convergence}\label{sec induct}
We pick the super-geometric sequence $\{\lambda_n\}$ to satisfy
\begin{equation}\label{eq lambda_n}
\lambda_{n+1} \gtrsim \lambda_n^5.
\end{equation}
Now we aim to establish the following estimates using an induction argument
\begin{equation}\label{eq induct est}
c_q F_n \le E_{n,\pm} \le F_n, \qquad \| \nabla \boldu_n \|, \| \nabla E_{n,\pm} \| \lesssim_{M} \lambda_n
\end{equation}
for some $c_q  \in (0,1)$, with some well-designed bounds $F_n$. 

\subsection{Bounds on $\boldu_n$}

From \eqref{eq u_n+1}, \eqref{eq est v^1}, and \eqref{eq est v^2} we find that
\begin{align*}
\begin{split}
|\boldu_{n+1} - \boldu_n^{\delta_n}| & \lesssim_{M}  \LV a_{n+1} \RV \LC 1 + \frac{\lambda_n |\nabla \boldu_n^{\delta_n}|}{\lambda_{n+1}} \RC 
\\&+ \frac{\LV \nabla a_{n+1} \RV}{\lambda_{n+1}} + a_{n+1}^2 \LC 1 + \frac{\lambda_n^2 |\nabla \boldu_n^{\delta_n}|}{\lambda_{n+1}} \RC + \frac{\lambda_n |a_{n+1}| |\nabla a_{n+1}|}{\lambda_{n+1}}.
\end{split}
\end{align*}
From the induction assumption \eqref{eq induct est} and the estimates on $\amp$ in Section \ref{sec amp}, it follows that
\begin{equation}\label{eq 2nd grad}
\begin{split}
& |\nabla \boldu_n^{\delta_n}|, |\nabla E_{n,\pm}^{\delta_n}| \lesssim_{M} \lambda_n, \qquad |\nabla^2 \boldu_n^{\delta_n}|, |\nabla^2 E_{n,\pm}^{\delta_n}|  \lesssim_{M} \frac{\lambda_n}{\delta_n}, \\
& |a_{n+1}| \lesssim \sqrt{F_n}, \qquad |\nabla a_{n+1}| \lesssim_{M} \frac{\lambda_n}{\sqrt{F_n}}, \qquad |\nabla^2 a_{n+1}| \lesssim_{M} \frac{\lambda_n}{\sqrt{F_n}} \LC \frac{1}{\delta_n} + \lambda_n \RC.
\end{split}
\end{equation}
By choosing 
\begin{equation*}
F_n \ge \frac{\lambda_n}{\lambda_{n+1}}
\end{equation*}
we obtain
\begin{equation}\label{est diff u}
|\boldu_{n+1} - \boldu_n^{\delta_n}| \le C_*  \sqrt{F_n} 
\end{equation}
for some constant $C_* = C_*(M) > 0$. Since $\boldu_0 = 0$, it then follows that 
\begin{proposition}\label{prop prelim bound u}
For $n\ge 1$, there is a choice for $\{ \lambda_n \}$ such that
\begin{equation}\label{eq bound u}
\|\boldu_n\| \le C_* \sum^{n-1}_{j = 0} \sqrt{F_j},
\end{equation}
where $\|\cdot\|$ is the $L^\infty$-norm. 
\end{proposition}

Choosing a summable sequence $\{F_n\}$, the above implies that $\{\boldu_n\}$ is bounded, which, combining with the estimate in Section \ref{subsec W}, yields that 
\begin{equation*}
|\boldW_{n+1}| \lesssim_{M} \LC a_{n+1}^2 \| \boldu_n \| + |a_{n+1}|^3 \RC \lesssim_{M} F_n \sum^{n}_{j = 0} \sqrt{F_j}.
\end{equation*}

\subsection{Bounds on $E_{n,\pm}$}
Choose $\delta_n$ so that
\begin{equation*}
\delta_n \lambda_n \le F_n \sum^{n}_{j = 0} \sqrt{F_j},
\end{equation*}
then we have
\begin{equation}\label{eq prelim est w}
|w_{n+1}| \le C_0 F_n \sum^{n}_{j = 0} \sqrt{F_j}
\end{equation}
for some constant $C_0 = C_0(M) > 0$. 

\begin{proposition}\label{prop bound E}
For $\ep < \frac12$, there exist suitable parameters $\beta, \gamma$, $F_n$ and $c_q  \in (0,1)$ such that if $E_{n,\pm}$ satisfies \eqref{eq induct est}, then the following estimate for $E_{n+1, \pm}$ holds:
\begin{equation}\label{eq est E_n+1}
c_q  F_{n+1} \le E_{n+1, \pm} \le F_{n+1}.
\end{equation}
\end{proposition}
\begin{proof}
We will divide the argument into the following three cases, according to the definition of $\amp$.

{\bf Case 1. $E_{n,\pm} \le \beta F_n$} where $\beta$ is introduced in Section \ref{sec amp}, then 
\[
\LC c_q - C_0 \sum^{n}_{j = 0} \sqrt{F_j} \RC F_n \le E_{n+1,\pm} = E_{n,\pm} + w_{n+1,\pm} \le \LC \beta + C_0 \sum^{n}_{j = 0} \sqrt{F_j} \RC F_n.
\]
Thus we need
\begin{equation}\label{eq cond alpha beta 1}
\beta + C_0 \sum^{n}_{j = 0} \sqrt{F_j} \le \frac{F_{n+1}}{F_n} \le 1 - \frac{1}{c_q } C_0 \sum^{n}_{j = 0} \sqrt{F_j}.
\end{equation}
We also need a requirement on $\kappa$ and $\beta$
\[
1 - \beta \ge \LC 1 + \frac{1}{c_q } \RC C_0 \sum^{n}_{j = 0} \sqrt{F_j}. 
\]

{\bf Case 2. $E_{n,\pm} \ge \gamma F_n$}. In this case we have from \eqref{eq est s} and \eqref{eq choice for a} that
\begin{equation*}
\begin{split}
\LC \frac{(1 - 2\ep)\gamma}{2 - 2\ep} - C_0 \sum^{n}_{j = 0} \sqrt{F_j} \RC F_n &\le E_{n+1,\pm}
 = \frac{1-s_k}{2} E_{n,\pm} + w_{n+1,\pm} \\&\le \LC \frac{1}{2 - 2\ep} + C_0 \sum^{n}_{j = 0} \sqrt{F_j} \RC F_n.
\end{split}
\end{equation*}
So we need
\begin{equation}\label{eq cond alpha beta 2}
\frac{1}{2 - 2\ep} + C_0 \sum^{n}_{j = 0} \sqrt{F_j} \le \frac{F_{n+1}}{F_n} \le \frac{1}{c_q } \LC \frac{(1 - 2\ep)\gamma}{2 - 2\ep} - C_0 \sum^{n}_{j = 0} \sqrt{F_j} \RC.
\end{equation}
The above also imposes the following condition
\[
\frac{(1 - 2\ep)\gamma}{(2 - 2\ep) c_q } - \frac{1}{2 - 2\ep} \ge \LC 1 + \frac{1}{c_q } \RC C_0 \sum^{n}_{j = 0} \sqrt{F_j}. 
\]

{\bf Case 3. $\beta F_n \le E_{n,\pm} \le \gamma F_n$}. Now we have
\begin{align*}
& E_{n+1,\pm} \ge \LC 1 - \frac{1}{2 - 2\ep} \RC E_{n,\pm} - |w_{n+1, \pm}| \ge \LC \frac{(1 - 2\ep)\beta}{2 - 2\ep} - C_0 \sum^{n}_{j = 0} \sqrt{F_j} \RC F_n, \\
& E_{n+1,\pm} \le E_{n,\pm} + |w_{n+1, \pm}| \le \LC \gamma + C_0 \sum^{n}_{j = 0} \sqrt{F_j} \RC F_n.
\end{align*}
Thus for 
\begin{equation}\label{eq cond alpha beta 3}
\gamma + C_0 \sum^{n}_{j = 0} \sqrt{F_j} \le \frac{F_{n+1}}{F_n} \le \frac{(1 - 2\ep)\beta}{(2 - 2\ep)c_q } - \frac{1}{c_q } C_0 \sum^{n}_{j = 0} \sqrt{F_j}.
\end{equation}
one would obtain \eqref{eq est E_n+1}. As in the previous cases, the above leads to assuming
\[
\frac{(1 - 2\ep)\beta}{(2 - 2\ep)c_q } - \gamma \ge \LC 1 + \frac{1}{c_q } \RC C_0 \sum^{n}_{j = 0} \sqrt{F_j}. 
\]

Now for some $0 < r < 1$, take
\[
F_n = \frac{\epsilon^2 r^{2n}}{C_0^2}
\]
for some $0 < \epsilon \ll 1$. Then
\[
\frac{F_{n+1}}{F_n} = r^2, \qquad C_0 \sum^\infty_{n=0} \sqrt{F_n} = \frac{\epsilon}{1- r}.
\]
If we consider 
\begin{equation}\label{para cond}
c_q < \frac{(1 - 2\ep) \beta}{2 - 2\ep} < \beta < \gamma < \frac{1}{2 - 2\ep}, \qquad c_q  < \frac{(1 - 2\ep) \beta}{(2 - 2\ep) \gamma},
\end{equation}
then \eqref{eq cond alpha beta 1}--\eqref{eq cond alpha beta 3} amount to asking
\[
\frac12 + \frac{\epsilon}{1 - r} \le r^2 \le 1 - \frac{\epsilon}{c_q  (1 - r)},
\]
which easily holds if we choose $r^2 > \frac12$ and $\epsilon$ sufficiently small. 
\end{proof}

Summarizing the above, we have obtained that 
\begin{equation}\label{est sum}
\boxed{
\|\boldu_n\|_{L^\infty} \le \frac{C_* \epsilon}{C_0(1-r)}, \qquad \frac{c_q  \epsilon^2 r^{2n}}{C_0^2} \le E_{n,\pm} \le \frac{\epsilon^2 r^{2n}}{C_0^2}, \qquad \delta_n \lambda_n \le \frac{r^{2n} \epsilon^3}{C_0^3 (1-r)}.
}
\end{equation}
Choosing sufficiently small $\epsilon$ we are able to achieve that $\|\boldu_n\|_{L^\infty} \le 1$, and hence we may take $M = 1$ in all of the preceding estimates.

\subsection{Bounds on $\|\nabla \boldu_n\|$}
Assuming \eqref{eq induct est}, it follows from \eqref{eq est v^1}, \eqref{eq est v^2}, and \eqref{eq 2nd grad}  that
\begin{equation}\label{eq est grad v}
\begin{split}
\| \nabla \boldv_{n+1}^1 \| & \lesssim \sqrt{F_n} \LC \lambda_{n+1} + \frac{\lambda_n^2}{\delta_n \lambda_{n+1}} \RC + \frac{\lambda_n}{\sqrt{F_n}} \LC 1 + \frac{1}{\delta_n \lambda_{n+1}} + \frac{\lambda_n}{\lambda_{n+1}} \RC \\
\| \nabla \boldv_{n+1}^2 \| & \lesssim F_n \LC \lambda_{n+1} + \frac{\lambda_n^3}{\delta_n \lambda_{n+1}} \RC + \frac{\lambda_n^2}{\lambda_{n+1}} \LC \lambda_{n+1} + \frac{1}{\delta_n} + \frac{\lambda_n}{F_n} \RC.
\end{split}
\end{equation}
By choosing
\begin{equation}\label{eq cond F delta}
\boxed{
\frac{\lambda_n^3}{\lambda_{n+1}^2} \le F_n \le r^{2n}, \qquad  \delta_n\lambda_n \ge \frac{\sqrt{\lambda_n}}{\lambda_{n+1}},
}
\end{equation}
it follows that
\begin{align*}
\| \nabla \boldu_{n+1} \| & \le \|\nabla \boldu_n\| + \|\nabla \boldv_{n+1}^1 \| + \|\nabla \boldv_{n+1}^2 \| \lesssim \lambda_{n+1}.
\end{align*}

\subsection{Bounds on $\|\nabla E_{n,\pm}\|$}
From the definition \eqref{eq E_n+1} we find that
\begin{equation*}
\nabla E_{n+1, \pm} = \nabla E_{n, \pm}^{\delta_n} - \frac{1+s_\pm}{2} \amp \nabla \amp - \frac{\nabla s_\pm}{4} \LC \amp \RC^2 + \nabla w_{n+1, \pm}.
\end{equation*}
From \eqref{eq 2nd grad} we know that
\[
|\nabla E_{n, \pm}^{\delta_n}| + |\amp \nabla \amp| + |\amp|^2 |\nabla s_\pm| \lesssim \lambda_n + \ep F_n (\lambda_{n+1} + \lambda_n^2).
\]
Recall \eqref{eq def w}. It follows that
\begin{align*}
|\nabla w_{n+1, \pm}| & \lesssim |\nabla \boldW_{n+1, \pm}| + \LV \nabla \LC f^{\delta_n}(\boldu_n) - f(\boldu_n^{\delta_n})  \RC \RV. 
\end{align*}
Similar to \eqref{eq comm}, we have
\begin{equation*}
\LN \nabla \LC {f^{\delta_n}(\boldu_n)} - f(\boldu_n^{\delta_n}) \RC \RN_{L^\infty} \lesssim {\delta_n} \| \nabla \boldu_n \|_{L^\infty} \lesssim \delta_n \lambda_n.
\end{equation*}

From \eqref{eq est grad err} and \eqref{eq est grad v} we see that
\begin{equation*}
\LV \nabla \err_1 \RV + \LV \nabla \err_2 \RV + \sum_{k,l = \pm} \LV \nabla \err_3^{k,l} \RV \lesssim \lambda_{n+1}.
\end{equation*}
The estimates in Section \ref{subsec W} yield that 
\begin{align*}
|\nabla \boldW_{n+1}| \lesssim \lambda_{n+1} + \LC 1 + \sqrt{F_n} \lambda_n \RC \frac{\sqrt{F_n} \lambda_n^2}{\delta_n \lambda_{n+1}} + \frac{1}{\sqrt{F_n}} + \frac{\lambda_n^2}{\sqrt{F_n} \lambda_{n+1}} + \frac{\lambda_n^3}{F_n \lambda_{n+1}}.
\end{align*}
Under the condition \eqref{eq cond F delta}, it follows that $|\nabla \boldW_{n+1}| \lesssim \lambda_{n+1}$, and hence
\[
|\nabla E_{n+1, \pm}| \lesssim \lambda_{n+1}.
\]

\subsection{A quick summary}\label{subsec summary}

What we have achieved in this section is the following.

For $\epsilon, c_q , r \in (0,1)$ satisfying 
\[
\frac12 + \frac{\epsilon}{1 - r} \le r^2 \le 1 - \frac{\epsilon}{ c_q (1 - r)},
\]
choose
\[
F_n = \frac{\epsilon^2 r^{2n}}{C_0^2},
\]
where $C_0$ is as in \eqref{eq prelim est w}. By choosing the two sequences $\{\lambda_n\}$, $\{\delta_n\}$ with
\[
\lambda_{n+1} \gtrsim \lambda_n^5, \qquad \frac{\lambda_n^3}{\lambda_{n+1}^2} \le F_n \le r^{2n}, \qquad  \frac{\sqrt{\lambda_n}}{\lambda_{n+1}} \le \delta_n\lambda_n \le \frac{r^{2n} \epsilon^3}{C_0^3 (1-r)},
\]
then the sequence of the iterative approximations $\{ (\boldu_n, E_{n,-}, E_{n,+}) \}$ enjoys the following property:
\begin{equation}\label{sum induct est}
\boxed{
\|\boldu_n\| \lesssim \epsilon, \qquad c_q  \frac{\epsilon^2 r^{2n}}{C_0^2} \le E_{n,\pm} \le \frac{\epsilon^2 r^{2n}}{C_0^2}, \qquad \| \nabla \boldu_n \|, \| \nabla E_{n,\pm} \| \lesssim \lambda_n.
}
\end{equation}

\subsection{Convergence to a weak solution}
 From \eqref{sum induct est} we see that there exists a subsequence of $\{ (\boldu_n, E_{n,-}, E_{n,+}) \}$, still denoted by $\{ (\boldu_n, E_{n,-}, E_{n,+}) \}$, such that as $n \to \infty$, 
 \begin{equation*}
 \boldu_n \rightharpoonup \boldu \text{ weak-*}, \qquad \text{and} \qquad E_{n,\pm} \to 0.
 \end{equation*}
 
 Moreover, by construction we know that each $\boldu_n$ is smooth, and from \eqref{est diff u} it follows that $\{\boldu_n\}$ is a Cauchy sequence in $C^0$. Therefore we in fact have
\[
\boldu_n \to \boldu \quad \text{in } C^0.
\]
Thus, $\boldu$ is a continuous weak solution to \eqref{eq cons}.

  Passing to this limit we find from \eqref{eq iter n} that
\[
\p_t \boldu + \p_x f(\boldu) = 0 \qquad \text{in } \mathcal{D}',
\]
that is, $\boldu$ is a weak solution to \eqref{eq cons}.

\section{Dephasing and non-uniqueness}\label{sec nonunique}

Sections \ref{sec subsoln approx}--\ref{sec induct} provide a systematic way to build weak solutions to system \eqref{eq cons} from a ``constant state'' subsolution of the form \eqref{const subsoln}. We would like to take advantage of the temporal phase function $P(t)$ in the approximation iteration \eqref{eq u_n+1} to generate two distinct weak solutions $\boldu^{(1)}$ and $\boldu^{(2)}$, sharing the same data at $t = 0$. 

For this, we first define a smooth function $\psi \in C^\infty(\R)$ such that 
\begin{equation*}
\psi(t) = \left\{\begin{aligned}
 & 0, & \quad & \text{ when } t \le -1, \\
 & \pi, & & \text{ when } t \ge -\frac12,
\end{aligned}\right. 
\end{equation*}
and $\psi$ is increasing. 

Starting with the same subsolution of the form \eqref{const subsoln} for $\alpha$ sufficiently small, say 
\begin{equation}\label{choice alpha}
\alpha = c_q  F_0 = c_q  \frac{\epsilon^2}{C_0^2}
\end{equation}
as in \eqref{sum induct est}, consider two iteration sequences $\left\{ \boldu^{(1)}_n \right\}$ and $\left\{ \boldu^{(2)}_n \right\}$ as in \eqref{eq u_n+1}, where the phase functions $P^{(1)}(t)$, $P^{(2)}(t)$ in the oscillatory parts $\boldv_{n+1}^{(1)}$ and $\boldv_{n+1}^{(2)}$ are given by
\[
  P^{(1)}(t) = 0, \qquad P^{(2)}(t) = \psi(t). 
\]
The discussion in the previous sections implies that 
\begin{equation}\label{conv to wk soln}
\boldu^{(i)}_n \to \boldu^{(i)} \quad \text{in } C^0
\end{equation}
with $\boldu^{(i)}$ being a weak solution of \eqref{eq cons}, for $i = 1,2$. We further choose the mollification scales $\delta_n$ to be sufficiently small such that 
\[
\sum^\infty_{n=0} \delta_n \le 1. 
\]

\begin{proposition}\label{prop ancient agree}
For all $n \ge 0$ it holds that 
\begin{equation}\label{eq ancient iter agree}
 \boldu^{(1)}_n =  \boldu^{(2)}_n, \ \ \ E^{(1)}_{n, \pm} = E^{(2)}_{n, \pm}  \qquad \textup{for} \quad t \le -1 - \sum^{n-1}_{k=0} \delta_k,
\end{equation}
where we consider $t\le -1$ when $n = 0$. In particular, this implies that
\begin{equation}\label{eq ancient agree}
 \boldu^{(1)} =  \boldu^{(2)} \qquad \textup{for} \quad t \le -2. 
\end{equation}
\end{proposition}
\begin{proof}
We will prove \eqref{eq ancient iter agree} by induction. Obviously \eqref{eq ancient iter agree} holds at $n = 0$.

Assume that \eqref{eq ancient iter agree} holds for some $n \ge 1$. By definition, for $t \le -1 - \sum^{n-1}_{k=0} \delta_k$ we have $P^{(1)}(t) = P^{(2)}(t)$, which immediately yields that 
\[
\eta_{\delta_n} \ast \boldu^{(1)}_n = \eta_{\delta_n} \ast \boldu^{(2)}_n \quad \text{for} \quad t \le -1 - \sum^{n}_{k=0} \delta_k.
\]
By construction and the definition of $\boldv_{n+1}$ it follows that 
\[
\boldu^{(1)}_{n+1} =  \boldu^{(2)}_{n+1}, \ \ \ E^{(1)}_{n, \pm} = E^{(2)}_{n, \pm} \qquad \textup{for} \quad t \le -1 - \sum^{n}_{k=0} \delta_k.
\]
Finally, \eqref{eq ancient agree} follows from the strong convergence \eqref{conv to wk soln} and the summability of $\delta_n$.
\end{proof}

Recall that now we have
\[
(\boldu^{(i)}_0, E^{(i)}_{0,+}, E^{(i)}_{0,-}) = (0, \alpha, \alpha), \qquad i = 1,2.
\]
We can prove that
\begin{proposition}\label{prop difference}
For all $n \ge 1$, when 
\[
\sqrt{t^2 + x^2} \le \frac{1}{2C \lambda_n} \quad \text{and} \quad \delta_n \le \frac{1}{2C \lambda_n}, 
\]
where
\begin{equation}\label{def C}
C = \frac{8\LC 1 + \sup\{ |\Lambda^+(0)|, |\Lambda^-(0)| \} \RC}{\pi},
\end{equation}
it follows that there exists some constant $C$ such that 
\begin{align}
& \LA \boldu^{(1)}_n, \rev_\pm (0) \RA, \ \LA \boldu^{(1), \delta_n}_n, \rev_\pm (0) \RA \ge \sqrt{\alpha} \LC 1 + p_0 \RC - C \sum^n_{j=1}\LC \frac{2}{c_q }\alpha r^{j-1} + \frac{1}{\lambda_j} \RC, \label{est proj first} \\
& \LA \boldu^{(2)}_n, \rev_\pm (0) \RA, \ \LA \boldu^{(2), \delta_n}_n, \rev_\pm (0) \RA \le -\sqrt{\alpha} \LC 1 + p_0 \RC + C \sum^n_{j=1}\LC \frac{2}{c_q }\alpha r^{j-1} + \frac{1}{\lambda_j} \RC, \label{est proj second}
\end{align}
where $r \in (0,1)$ is in Section \ref{subsec summary}, and $p_0 \in [0,1)$ is defined in \eqref{eigenvalue inner product}.
\end{proposition}
\begin{proof}
We will use an induction argument to prove \eqref{est proj first} and \eqref{est proj second}. Consider first the case when $n = 1$. For any $\delta_0 > 0$, 
\[
\boldu^{(i), \delta_0}_0 = 0, \quad E^{(i), \delta_0}_{0,\pm} = \alpha, \quad  \Lambda^{(i), \pm}_0 = \Lambda^{\pm}(0), \qquad i = 1,2,
\]
and thus
\[
a^{(i), \pm}_1 = \sqrt{2\alpha} \phi_{\beta, \gamma} (\alpha/F_0) = \sqrt{2\alpha} \phi_{\beta, \gamma}(c_q ).
\]
From the condition \eqref{para cond} on the parameters we can choose appropriate $\beta, \gamma$, and the function $\phi_{\beta, \gamma}$ such that $\phi_{\beta, \gamma}(c_q ) = 1$, and so 
\[
a^{(i), \pm}_1 = \sqrt{2\alpha}, \qquad i = 1,2.
\]

From \eqref{eq u_n+1}, \eqref{eq v^1}, and \eqref{eq v^2} we know that
\begin{align*}
\boldu^{(1)}_1 = & \ \sqrt{2\alpha} \sum_j \varphi_{0,j}(\Lambda^+(0)) \cos \LB \lambda_1 (x - \Lambda^+_{0,j} t) + P^{(1)}(t) \RB \rev_+(0) \\
& \ + \sqrt{2\alpha} \sum_j \varphi_{0,j}(\Lambda^-(0)) \cos \LB \lambda_1 (x - \Lambda^-_{0,j} t) + P^{(1)}(t) \RB \rev_-(0) + \mathcal{R}^{(1)}_1, 
\end{align*}
where $|\mathcal{R}^{(1)}_1| \le C \LC \alpha + \frac{1}{\lambda_1} \RC$.

Recall the definition of the localization in Section \ref{subsec loc}. We may choose $\lambda_0$ sufficiently large such that there exists only one $j_+$ such that $\varphi_{0,j_+}(\Lambda^+(0)) \ne 0$. The partition of unity further implies that at such $j$, $\varphi_{0,j_+}(\Lambda^+(0)) = 1$. Similar result holds for $\varphi_{0,j}(\Lambda^-(0))$. Therefore we have 
\begin{align*}
\boldu^{(1)}_1 = & \ \sqrt{2\alpha} \LCB \cos \LB \lambda_1 (x - \Lambda^+_{0,j_+} t) \RB \rev_+(0) + \cos \LB \lambda_1 (x - \Lambda^-_{0,j_-} t) \RB \rev_-(0) \RCB + \mathcal{R}^{(1)}_1,
\end{align*}
with
\[
\LV \Lambda^{\pm}_{0, j_\pm} - \Lambda^\pm(0) \RV \le \frac{2}{3\lambda_0}.
\]
Similarly, for $\boldu^{(2)}_1$ we have
\begin{align*}
\boldu^{(2)}_1 = & \ \sqrt{2\alpha} \LCB \cos \LB \lambda_1 (x - \Lambda^+_{0,j_+} t) + P^{(2)}(t) \RB \rev_+(0) + \cos \LB \lambda_1 (x - \Lambda^-_{0,j_-} t) + P^{(2)}(t) \RB \rev_-(0) \RCB \\
& \ + \mathcal{R}^{(2)}_1 
\end{align*}
with $|\mathcal{R}^{(2)}_1| \le C \LC \alpha + \frac{1}{\lambda_1} \RC$. 

For $|t| < \frac{1}{2}$ we know that $P^{(2)}(t) = \pi$ and the above becomes
\[
\boldu^{(2)}_1 = - \sqrt{2\alpha} \LCB \cos \LB \lambda_1 (x - \Lambda^+_{0,j_+} t) \RB \rev_+(0) + \cos \LB \lambda_1 (x - \Lambda^-_{0,j_-} t) \RB \rev_-(0) \RCB + \mathcal{R}^{(2)}_1.
\]
Note that 
\[
\LV \lambda_1 \LC x - \Lambda^{\pm}_{0, j_\pm} t \RC \RV \le 2 \lambda_1 \LC 1 + \sup\{ |\Lambda^+(0)|, |\Lambda^-(0)| \} \RC \sqrt{t^2 + x^2}.
\]
Therefore
\[
\sqrt{t^2 + x^2} \le \frac{1}{C \lambda_1} \quad \Rightarrow \quad \LV \lambda_1 \LC x - \Lambda^{\pm}_{0, j_\pm} t \RC \RV \le \frac{\pi}{4} \quad \Rightarrow \quad \cos \LB \lambda_1 (x - \Lambda^\pm_{0,j_\pm} t) \RB \ge \frac{\sqrt{2}}{2},
\]
where $C$ is given in \eqref{def C}. This way, if we choose 
\[
\delta_1 \le \frac{1}{2 C \lambda_1},
\]
then
\[
\sqrt{t^2 + x^2} \le \frac{1}{2 C \lambda_1} \quad \Rightarrow \quad \eta^{\delta_1} \ast \cos \LB \lambda_1 (x - \Lambda^\pm_{0,j_\pm} t) \RB \ge \frac{\sqrt{2}}{2}.
\]
Choose $\alpha$ and $\lambda_1$ sufficiently small, and
hence in this region
\begin{align*}
& \LA \boldu^{(1)}_1, \rev_\pm(0) \RA, \ \LA \boldu^{(1), \delta_1}_1, \rev_\pm(0) \RA \ge  \sqrt{\alpha} \LC 1 + p_0 \RC - C \LC \alpha + \frac{1}{\lambda_1} \RC, \\
& \LA \boldu^{(2)}_1, \rev_\pm(0) \RA, \ \LA \boldu^{(2), \delta_1}_1, \rev_\pm(0) \RA \le -\sqrt{\alpha} \LC 1 + p_0 \RC + C \LC \alpha + \frac{1}{\lambda_1} \RC,
\end{align*}
which proves \eqref{est proj first} and \eqref{est proj second} for the case $n = 1$.

Assume now that \eqref{est proj first} and \eqref{est proj second} hold for some general $n \ge 1$. 
When $|t| < \frac{1}{2}$ we have
\begin{align*}
\boldu^{(1)}_{n+1} = & \ \boldu^{(1), \delta_n}_{n} +  \sum_\pm \sum_j \varphi_{n,j}\LC \ev_n^{(1), \pm} \RC a^{(1),\pm}_{n+1} \cos\LB \lambda_{n+1} (x - \ev_{n,j}^{\pm} t) \RB \rev_{\ev_n^{(1), \pm}} + \mathcal{R}^{(1)}_{n+1}, \\
\boldu^{(2)}_{n+1} = & \ \boldu^{(2), \delta_n}_{n} -  \sum_\pm \sum_j \varphi_{n,j}\LC \ev_n^{(2), \pm} \RC a^{(2),\pm}_{n+1} \cos\LB \lambda_{n+1} (x - \ev_{n,j}^{\pm} t) \RB \rev_{\ev_n^{(2), \pm}} + \mathcal{R}^{(2)}_{n+1},  
\end{align*}
where $|\mathcal{R}^{(i)}_{n+1}| \le C \LC a_{n+1}^2 + \frac{1}{\lambda_{n+1}} \RC$ and $\rev_{\ev_n^{(i), \pm}}$ are the right eigenvectors of $\ev_n^{(i), \pm}$, for $i = 1,2$. 
From \eqref{est sum} and \eqref{choice alpha} we see that
\begin{equation}\label{est eigen}
\ev_n^{(i), \pm} = \ev^\pm(0) + O(\sqrt\alpha), \quad \text{and hence} \quad \rev_{\ev_n^{(i), \pm}} = \rev_\pm(0) + O(\sqrt\alpha).
\end{equation}
Similarly as before, by choosing $\lambda_n$ sufficiently large, it holds that for any $i = 1,2$, $k \in \{+,-\}$, there exists a unique $j^{(i)}_k \in \mathbb{Z}$ such that 
\begin{align*}
\boldu^{(i)}_{n+1} = & \ \boldu^{(i), \delta_n}_{n} + (-1)^{i-1} \sum_{k = \pm} a^{(i),k}_{n+1} \cos\LB \lambda_{n+1} (x - \ev_{n,j^{(i)}_k}^{k} t) \RB \rev_{\ev_n^{(i), k}} + \mathcal{R}^{(i)}_{n+1}. \end{align*}
From the definition of $\varphi_{n,j}$ in Section \ref{subsec loc} we know that
\[
\LV \ev_{n,j^{(i)}_k}^{k} - \ev^{(i), k}_n \RV \le \frac{2}{3\lambda_n}, \quad \text{and hence} \quad \LV \ev_{n,j^{(i)}_k}^{k} - \ev^{k}(0) \RV \le \frac{2}{3\lambda_n} + O(\sqrt\alpha).
\]
Therefore
\[
\sqrt{t^2 + x^2} \le \frac{1}{C \lambda_{n+1}} \quad \Rightarrow \quad \LV \lambda_{n+1} \LC x - \ev_{n,j^{(i)}_k}^{k} t \RC \RV \le \frac{\pi}{4} \quad \Rightarrow \quad \cos \LB \lambda_1 (x - \ev_{n,j^{(i)}_k}^{k} t) \RB \ge \frac{\sqrt{2}}{2},
\]
where $C$ is given in \eqref{def C}. Taking $\delta_{n+1} \le \frac{1}{2 C \lambda_{n+1}}$ we have
\[
\sqrt{t^2 + x^2} \le \frac{1}{2 C \lambda_{n+1}} \quad \Rightarrow \quad \eta^{\delta_{n+1}} \ast \cos \LB \ev_1 (x - \ev_{n,j^{(i)}_k}^{k} t) \RB \ge \frac{\sqrt{2}}{2}.
\]
Using \eqref{est eigen} we have
\begin{align*}
\LA \boldu^{(1)}_{n+1}, \rev_+(0) \RA = & \ \LA \boldu^{(1), \delta_n}_n, \rev_+(0) \RA + \sum_{k = \pm} a^{(1),k}_{n+1} \cos\LB \lambda_{n+1} (x - \ev_{n,j^{(1)}_k}^{k} t) \RB \LA \rev_{\ev_n^{(1), k}}, \rev_+(0) \RA \\
& \ + \LA \mathcal{R}^{(1)}_{n+1}, \rev_+(0) \RA \\ 
= & \ \LA \boldu^{(1), \delta_n}_n, \rev_+(0) \RA + a^{(1),+}_{n+1} \cos\LB \lambda_{n+1} (x - \ev_{n,j^{(1)}_+}^{+} t) \RB \\
& \ + p_0 a^{(1),-}_{n+1} \cos\LB \lambda_{n+1} (x - \ev_{n,j^{(1)}_-}^{-} t) \RB + \LA \mathcal{R}^{(1)}_{n+1}, \rev_+(0) \RA \\
& \ + \underbrace{a^{(1),-}_{n+1} \cos\LB \lambda_{n+1} (x - \ev_{n,j^{(1)}_-}^{-} t) \RB \LA \rev_{\ev_n^{(1), -}} - \rev_-(0) , \rev_+(0) \RA}_{= \ O \LC a_{n+1} \sqrt\alpha \RC}.  
\end{align*}
From \eqref{eq choice for a} we know that $0 \le a_{n+1} \le \sqrt{E_{n, \pm}^{\delta_n}}$. From \eqref{sum induct est} we see that $E_{n, \pm} \le \alpha r^{2n}/c_q $. Hence
\[
0 \le a_{n+1} \le r^n \sqrt{\alpha/c_q },
\]
and therefore
\[
a_{n+1}^2 + a_{n+1}\sqrt\alpha = \alpha r^n \LC \frac{r^n}{c_q } + \frac{1}{\sqrt c_q } \RC < \frac{2}{c_q } \alpha r^n.
\]
So when $\sqrt{t^2 + x^2} \le \frac{1}{2 C \lambda_{n+1}}$, from the induction assumption,
\begin{align*}
\LA \boldu^{(1)}_{n+1}, \rev_+(0) \RA \ge & \ \ \LA \boldu^{(1), \delta_n}_n, \rev_+(0) \RA + \frac{\sqrt2}{2} \LC a^{(1),+}_{n+1} + p_0 a^{(1),-}_{n+1} \RC - C \LC \frac{2}{c_q } \alpha r^n + \frac{1}{\lambda_{n+1}} \RC \\
\ge & \ \sqrt{\alpha} \LC 1 + p_0 \RC - C \sum^{n+1}_{j=1}\LC \frac{2}{c_q }\alpha r^{j-1} + \frac{1}{\lambda_j} \RC.
\end{align*}
The rest of the estimates can be obtained through the same argument. 
\end{proof}

Now we can state our main non-uniqueness result of this section.
\begin{theorem}\label{thm nonunique}
The two continuous weak solutions $\boldu^{(1)}, \boldu^{(2)}$ of system \eqref{eq cons} constructed through the process in Proposition \ref{prop difference} have the property that
\[
\boldu^{(1)} = \boldu^{(2)} \ \text{for } \ t \le -2, \quad \text{but} \quad \boldu^{(1)}(0,0) \ne \boldu^{(2)}(0,0).
\]
\end{theorem}
\begin{proof}
The agreement of $\boldu^{(1)}$ and $\boldu^{(2)}$ on $t\le -2$ has be proved in Proposition \ref{prop ancient agree}. From Proposition \ref{prop difference} we see that at $(t,x) = (0,0)$, by choosing $\alpha$ small enough,
\begin{align*}
\LA \boldu^{(1)}_{n} - \boldu^{(2)}_{n}, \rev_+(0) \RA \ge 2 \sqrt{\alpha} \LC 1 + p_0 \RC - 2 C \sum^n_{j=1}\LC \frac{2}{c_q }\alpha r^{j-1} + \frac{1}{\lambda_j} \RC \ge \sqrt{\alpha} \LC 1 + p_0 \RC
\end{align*}
for $n$ sufficiently large. Sending $n \to \infty$ yields
\[
\LA \boldu^{(1)}_{n}(0,0) - \boldu^{(2)}_{n}(0,0), \rev_+(0) \RA \ge \sqrt{\alpha} \LC 1 + p_0 \RC,
\]
which concludes the theorem. 
\end{proof}

Once Theorem \ref{thm nonunique} is proved, our main result Theorem \ref{thetheo} follows.

\section{Proof of Theorem \ref{theo-ex}}\label{sec example}
In this section, we study System \ref{example}  that substantiates our hypothesis regarding the non-uniqueness of continuous solutions within the context of a 1D system of conservation laws, and prove Theorem \ref{theo-ex}.
Letting 
$U=\begin{pmatrix} u\\ v
	\end{pmatrix}$ and $f(U)=\begin{pmatrix} \frac{u v}{2}+v\\ u-\frac{v^2}{2}
	\end{pmatrix}$, one calculates
	 $$Df(U)=\begin{pmatrix} \frac{v}{2}\ &\frac{u}{2}+1\\1 \ &-v
	\end{pmatrix}.$$
Computing the trace and the determinant of this matrix, we find that 
$$
\Lambda^-+\Lambda^+=-v/2, \qquad \Lambda^-\Lambda^+=-v^2/2-(1+u/2).
$$	
Hence, System \eqref{example} is strictly hyperbolic on $\left\{(u,v) : \phi(u,v)=4+2u+(9/4)v^2>0 \right\}$, and on this set 
$$
\Lambda^{\pm}=-\frac{v}{4}\pm\frac{\sqrt{\phi}}{2}.
$$
The associated eigenvectors are 
$$
\rev_\pm=\begin{pmatrix} \pm \frac{3v}{4}\pm\frac{\sqrt{\phi}}{2}\\ 1
	\end{pmatrix}.
$$	
We now compute:
$$
\nabla\Lambda^\pm=\left(\pm\frac{1}{2\sqrt{\phi}},-\frac{1}{4}\pm\frac{9v}{8\sqrt{\phi}}\right),
$$	
and so 
\begin{equation}\label{evlambda}
(\rev_\pm\cdot\nabla)\Lambda^\pm=\pm \frac{3v}{2\sqrt{\phi}}. 
\end{equation}
Note that at 0, these quantities are equal to 0. Therefore, to verify $\mathcal{C}_\ep$ of Definition \ref{cond}, we only need to show that the curvatures $\kappa_\pm$ are not 0.

We are able to  calculate $A=|\det(\rev_-,\rev_+)|=2 $,
and $\delta_\Lambda=\Lambda^{+}-\Lambda^{-}=\sqrt{\phi(0,0)}=2>0$ at $(0,0)$.
From \eqref{eq b}, this would imply that $\boldb_\pm(0)=0$ if $\kappa_\pm=0$ at $(0,0).$ 
But we can calculate that 
	$$D^2f(0)=\begin{pmatrix}\begin{pmatrix} 0 &\frac{1}{2}\\\frac{1}{2} &0\end{pmatrix}\\

	\begin{pmatrix} 0 &0\\0 &-1\end{pmatrix}
	\end{pmatrix},$$
which implies that $$D^2f(0):(\rev_{\pm}(0)\otimes \rev_{\pm}(0))=\frac{1}{2}\begin{pmatrix}\pm 1\\-1
\end{pmatrix} = \boldb_{\pm}\neq0.
$$
Therefore, System \ref{example} verifies $\mathcal{C}_\ep$ for all $0<\ep<1$ at 0. From Theorem \ref{thetheo}, there exists $\rho > 0$ such that for any ball $B\subset B(0,\rho)$, there exists two solutions 
of \eqref{example} in $C^0(\R^+\times\R;B)$ sharing the same initial value. 
Choosing such a ball $B$ which does not intersect the line $\{v=0\}$, we see from \eqref{evlambda} that, in addition, both fields are genuinely nonlinear in $B$. This ends the proof of Theorem \ref{theo-ex}.
\vskip0.3cm
Note that  this system has an entropy $\eta$ as $$\eta(U)=\frac{u^2}{2}+ \LC 1+\frac{u}{2} \RC \frac{v^2}{2}-\frac{v^4}{16}$$
 in term of $U=(u,v).$
We can verify that $$\eta''=\begin{pmatrix} 1 & \frac{v}{2}\\\frac{v}{2} &(1+\frac{u}{2})-\frac{3}{4}v^2
	\end{pmatrix}.$$
	The trace of this matrix is given by $2+\frac{u}{2}-\frac{3}{4}v^2>0$ for any $|(u,v)|\ll 1$,  and its determinant is given by $1+\frac{u}{2}-v^2>0$  for any $|(u,v)|\ll 1$.
	Hence, $\eta$ is a convex function around $(0,0)$. This assertion underscores the suitability of our system \eqref{example} as a good system. 

\appendix

\section{Calculation and estimates on the correctors}\label{sec appendix correct}

In this appendix we collect the explicit computation involved in Section \ref{sec subsoln approx}, together with the remainder estimates.  

First, we have
\begin{align}\label{eq remainder 1}
& \p_t \boldv^{1,\pm}_{n+1} + \p_x \LC \ev_n^\pm \boldv^{1,\pm}_{n+1} \RC \\
= & \  \p_x \LCB \frac{1}{\lambda_{n+1}}\sum_j \p_t \big[ \varphi_{n,j} \LC \ev_n^\pm \RC \amp  \rev_n^\pm \big] \sin \LB \lambda_{n+1} (x - \ev_{n,j}^\pm t) + P(t) \RB \RCB \nonumber \\
& \ + \p_x \LCB \sum_j \varphi_{n,j} \LC \ev_n^\pm \RC \amp  \rev_n^\pm \cos \LB \lambda_{n+1} (x - \ev_{n,j}^\pm t) + P(t) \RB \LC - \ev_{n,j}^\pm + \frac{P'}{\lambda_{n+1}} \RC  \RCB \nonumber\\
& \ + \p_x \LCB \frac{1}{\lambda_{n+1}} \sum_j  \ev_n^\pm \p_x \big[ \varphi_{n,j} \LC \ev_n^\pm \RC \amp \rev_n^\pm \big] \sin \LB \lambda_{n+1} (x - \ev_{n,j}^\pm t) + P(t) \RB \RCB \nonumber \\
& \ + \p_x \LCB \sum_j \varphi_{n,j} \LC \ev_n^\pm \RC \amp  \rev_n^\pm \cos \LB \lambda_{n+1} (x - \ev_{n,j}^\pm t) + P(t) \RB \ev_n^\pm  \RCB \nonumber \\
= & \ \p_x \LCB \frac{1}{\lambda_{n+1}} \sum_j  \LC \p_t + \ev_n^\pm \p_x \RC\big[ \varphi_{n,j} \LC \ev_n^\pm \RC \amp \rev_n^\pm \big] \sin \LB \lambda_{n+1} (x - \ev_{n,j}^\pm t) + P(t) \RB \RCB \nonumber \\
& \ + \p_x \LCB \sum_j \varphi_{n,j} \LC \ev_n^\pm \RC \amp \rev_n^\pm \cos \LB \lambda_{n+1} (x - \ev_{n,j}^\pm t) + P(t) \RB \LC \ev_n^\pm - \ev_{n,j}^\pm + \frac{P'}{\lambda_{n+1}} \RC  \RCB \nonumber \\
=: & \ \p_x \rem^{(1),\pm}_{n+1}. \nonumber
\end{align}
Since $|\varphi_{n,j}^{(m)}| \lesssim \lambda_n^m$ and $P'$ is bounded, from \eqref{eq est mean osc} we see that
\begin{equation}\label{eq est R_1}
\begin{aligned}
& \LV \rem^{(1),\pm}_{n+1} \RV & \lesssim_{M} & \LV a_{n+1} \RV \LC \frac{ \lambda_n |\nabla \boldu_n^{\delta_n}|}{\lambda_{n+1}} + \frac{1}{\lambda_n} \RC + \frac{\LV \nabla a_{n+1} \RV}{\lambda_{n+1}}, \\
& \LV \nabla \rem^{(1),\pm}_{n+1} \RV & \lesssim_{M} & \LV a_{n+1} \RV \LC \frac{\lambda_{n+1}}{\lambda_n} + \lambda_n |\nabla \boldu_n^{\delta_n}| +  \frac{\lambda_n^2 |\nabla \boldu_n^{\delta_n}|^2 + \lambda_n {|\nabla^2 \boldu_n^{\delta_n}|}}{\lambda_{n+1}} \RC + \\
& & & \LV \nabla a_{n+1} \RV \LC \frac{\lambda_n |\nabla \boldu_n^{\delta_n}|}{\lambda_{n+1}} + \frac{1}{\lambda_n} \RC + \frac{{\LV \nabla^2 a_{n+1} \RV}}{\lambda_{n+1}}.
\end{aligned}
\end{equation}

Recall that 
\begin{equation*}
\rem^{(2),\pm}_{n+1} := \LB Df(\boldu_n^{\delta_n}) - \ev_n^\pm \id \RB \boldv_{n+1}^{1, \pm}.
\end{equation*}
Using the fact that $\LB Df(\boldu_n^{\delta_n}) - \ev_n^\pm \id \RB \rev_n^\pm = 0$, an improved estimate can be obtained.
\be\label{eq est R_2}
\begin{aligned}
& \LV \rem^{(2),\pm}_{n+1} \RV & \lesssim_{M} & \ \frac{\LV a_{n+1} \RV |\nabla \boldu_n^{\delta_n}|}{\lambda_{n+1}}, \\
& \LV \nabla \rem^{(2),\pm}_{n+1} \RV & \lesssim_{M} & \ \LV a_{n+1} \RV |\nabla \boldu_n^{\delta_n}| \LC 1 + \frac{\lambda_n |\nabla \boldu_n^{\delta_n}|}{\lambda_{n+1}} \RC + \frac{\LV \nabla a_{n+1} \RV |\nabla \boldu_n^{\delta_n}|}{\lambda_{n+1}}.
\end{aligned}
\ee

As for $\rem^{(3),k,l}_{n+1}$, we have 
\begin{align}\label{eq def R_3}
\rem^{(3),k,l}_{n+1} := & \LC D^2f(\boldu_n^{\delta_n}) - D^2f(0) \RC : \LC \boldv_{n+1}^{1, k} \otimes \boldv_{n+1}^{1, l} \RC  \\
&\ + \sum_{i,j} \sin \LB \lambda_{n+1} (x - \ev_{n,i}^k t) + P(t) \RB \sin \LB \lambda_{n+1} (x - \ev_{n,j}^l t) + P(t) \RB \cdot \nonumber \\
& \ \quad \frac{a_{n+1}^k \cdot a_{n+1}^l}{\lambda_{n+1}^2} \LB D^2f(0) : \LC \p_x \LC \varphi_{n,j} \LC \ev_n^k \RC \rev_n^k \RC \otimes \p_x \LC \varphi_{n,j} \LC \ev_n^l \RC \rev_n^l \RC \RC \RB + \nonumber \\
& \ + \sum_{i,j} \varphi_{n,i} \LC \ev_n^k \RC \varphi_{n,j} \LC \ev_n^l \RC \cos \LB \lambda_{n+1} (x - \ev_{n,i}^k t) + P(t) \RB \cos \LB \lambda_{n+1} (x - \ev_{n,j}^l t) + P(t) \RB \cdot \nonumber \\
& \ \quad \LC a_{n+1}^k \cdot a_{n+1}^l \RC \LB D^2f(0) : \LC (\rev_n^k - \rev^k ) \otimes ( \rev_n^l - \rev^l ) \RC \RB. \nonumber 
\end{align}
This way using \eqref{eq est v^1} we obtain the estimate
{\small\begin{align}\label{eq est R_3}
& \LV \rem^{(3),k,l}_{n+1} \RV & \lesssim_{M} & \ a_{n+1}^2 \LB |\boldu_n^{\delta_n}|  \LC 1 + \frac{\lambda_n |\nabla \boldu_n^{\delta_n}|}{\lambda_{n+1}} \RC^2 + \LC \frac{\lambda_n |\nabla \boldu_n^{\delta_n}|}{\lambda_{n+1}} \RC^2 + |\boldu_n^{\delta_n}|^2 \RB + \frac{\LV \nabla a_{n+1} \RV^2 |\boldu_n^{\delta_n}|}{\lambda_{n+1}^2},\nonumber \\
& \LV \nabla \rem^{(3),k,l}_{n+1} \RV & \lesssim_{M} & \ |\nabla \boldu_n^{\delta_n}| \LV \boldv_{n+1}^{1,\pm} \RV^2 + |\boldu_n^{\delta_n}| \LV \boldv_{n+1}^{1,\pm} \RV {\LV \nabla \boldv_{n+1}^{1,\pm} \RV} + a_{n+1}^2 |\boldu_n^{\delta_n}|^2 \LC \lambda_n |\nabla \boldu_n^{\delta_n}| + \lambda_{n+1} \RC + \nonumber \\
& & & \ a_{n+1}^2 |\boldu_n^{\delta_n}| |\nabla \boldu_n^{\delta_n}| + \frac{a_{n+1}^2 \lambda_n^2 |\nabla \boldu_n^{\delta_n}|}{\lambda_{n+1}} \LC |\nabla \boldu_n^{\delta_n}| + \frac{\lambda_n |\nabla \boldu_n^{\delta_n}|^2 + {|\nabla^2 \boldu_n^{\delta_n}|}}{\lambda_{n+1}} \RC + \nonumber \\
& & & \ |a_{n+1}| |\nabla a_{n+1}| \LC |\boldu_n^{\delta_n}|^2 + \frac{\lambda_n^2 |\nabla \boldu_n^{\delta_n}|^2}{\lambda_{n+1}^2} \RC.
\end{align}}

The remainder $\rem^{(4),k,k}_{n+1}$ is defined as 
{\small
\begin{align}\label{eq def R_4 1}
\rem^{(4),k,k}_{n+1} := & \ \sum_j \frac{\boldB_k}{8\lambda_{n+1}} (\p_t + \ev_n^k \p_x) \LCB \LC a_{n+1}^k \RC^2 \LB \varphi_{n,j} \LC \ev_n^k \RC \RB^2 \sin \LB 2\lambda_{n+1} (x - \ev_{n,j}^k t) + 2P(t) \RB \RCB + \nonumber \\
& \ \sum_{|i-j| = 1} \frac{\boldB_k}{8\lambda_{n+1}} (\p_t + \ev_n^k \p_x) \LCB \LC a_{n+1}^k \RC^2  \varphi_{n,i} \LC \ev_n^k \RC \varphi_{n,j} \LC \ev_n^k \RC \sin \LB \lambda_{n+1} (2x - (\ev_{n,i}^k + \ev_{n,j}^k) t) + 2P(t) \RB \RCB \\&+ \nonumber  \ \sum_{|i-j| = 1} \frac{\boldb_k \sin \LC \lambda_{n+1}  (\ev_{n,j}^k - \ev_{n,i}^k) t \RC}{4 \lambda_{n+1} (\ev_{n,j}^k - \ev_{n,i}^k)} (\p_t + \ev_n^k \p_x) \LB \LC a_{n+1}^k \RC^2  \varphi_{n,i} \LC \ev_n^k \RC \varphi_{n,j} \LC \ev_n^k \RC \RB,
\end{align}}
and 
{\small
\begin{align}\label{eq def R_4 2}
\rem^{(4),+,-}_{n+1} := & \ \sum_{i \ne j} \frac{\boldD}{8\lambda_{n+1}} \LC \p_t + \frac{\ev_n^+ + \ev_n^-}{2} \p_x\RC \Big\{ \LC a_{n+1}^+ \cdot a_{n+1}^- \RC  \varphi_{n,i} \LC \ev_n^+ \RC \varphi_{n,j} \LC \ev_n^- \RC  \nonumber \\
& \ \qquad \left. \sin \LB \lambda_{n+1} (2x - (\ev_{n,i}^+ + \ev_{n,j}^-) t) + 2P(t) \RB \RCB +  \\
& \ \sum_{i \ne j} \frac{\boldd \sin \LC \lambda_{n+1}  (\ev_{n,j}^+ - \ev_{n,i}^-) t \RC}{4 \lambda_{n+1} (\ev_{n,j}^+ - \ev_{n,i}^-)} \LC \p_t + \frac{\ev_n^+ + \ev_n^-}{2} \p_x\RC \LB \LC a_{n+1}^+ \cdot a_{n+1}^- \RC  \varphi_{n,i} \LC \ev_n^+ \RC \varphi_{n,j} \LC \ev_n^- \RC \RB. \nonumber
\end{align}}
The estimates of $\rem^{(4),k,l}_{n+1}$ are as below.
\begin{align}\label{eq est R_4}
& \LV \rem^{(4),k,l}_{n+1} \RV & \lesssim_{M} & \ a_{n+1}^2 \LC \frac{ \lambda_n^2 |\nabla \boldu_n^{\delta_n}|}{\lambda_{n+1}} + \frac{1}{\lambda_{n}} \RC + \frac{\lambda_n |a_{n+1}| |\nabla a_{n+1}|}{\lambda_{n+1}}, \nonumber \\
& \LV \nabla \rem^{(4),k,l}_{n+1} \RV & \lesssim_{M} & \ \LV a_{n+1} \RV^2 \LC \frac{\lambda_{n+1}}{\lambda_n} + \lambda_n^2 |\nabla \boldu_n^{\delta_n}| +  \frac{\lambda_n^3 |\nabla \boldu_n^{\delta_n}|^2 + \lambda_n^2 |\nabla^2 \boldu_n^{\delta_n}|}{\lambda_{n+1}} \RC + \\
& & & \ |a_{n+1}| \LV \nabla a_{n+1} \RV \LC \frac{\lambda_n^2 |\nabla \boldu_n^{\delta_n}|}{\lambda_{n+1}} + \lambda_n \RC + \frac{\lambda_n}{\lambda_{n+1}} \LC |\nabla a_{n+1}|^2 + |a_{n+1}| |\nabla^2 a_{n+1}| \RC. \nonumber
\end{align}

The definition of $\rem^{(5),k,l}_{n+1} $, for $k, l \in \{ +, - \}$, together with the corresponding estimates, are given as 
\begin{align}\label{eq def R_5 1}
\rem^{(5),k,k}_{n+1} = & - \sum_{j} \frac{ \sin \LC 2\lambda_{n+1} (x - \ev_{n,j}^k t) + 2P(t) \RC }{8 \lambda_{n+1}} \p_x \LB \LC a_{n+1}^k \varphi_{n,j} \LC \ev_n^k \RC \RC^2 \RB  \tilde{\boldB}_k \\
&  - \sum_{|i-j| = 1} \frac{ \sin \LC \lambda_{n+1} (2x - (\ev_{n,i}^k + \ev_{n,j}^k) t) + 2P(t) \RC }{8 \lambda_{n+1}} \p_x \LB \LC a_{n+1}^k \RC^2 \varphi_{n,i} \LC \ev_n^k \RC \varphi_{n,j} \LC \ev_n^k \RC \RB  \tilde{\boldB}_k \nonumber \\
&  - \sum_{|i-j| = 1} \frac{ \sin \LC \lambda_{n+1}  (\ev_{n,j}^k - \ev_{n,i}^k) t \RC }{4 \lambda_{n+1} (\ev_{n,j}^k - \ev_{n,i}^k)} \p_x \LB \LC a_{n+1}^k \RC^2 \varphi_{n,i} \LC \ev_n^k \RC \varphi_{n,j} \LC \ev_n^k \RC \RB \LC Df(0) - \ev^k(0) \id \RC \tilde{\boldB}_k, \nonumber
\end{align}
and
{\small\begin{align}\label{eq def R_5 2}
\rem^{(5), +,-}_{n+1} = &  - \sum_{j} \frac{ \sin \LB \lambda_{n+1} (2x - (\ev_{n,i}^+ + \ev_{n,j}^-) t) + 2P(t) \RB }{8 \lambda_{n+1}} \p_x \LB \LC a_{n+1}^+ \cdot a_{n+1}^- \RC \varphi_{n,i} \LC \ev_n^+ \RC \varphi_{n,j} \LC \ev_n^- \RC \RB  \boldd \nonumber \\
& - \sum_{j} \frac{ \sin \LB \lambda_{n+1} (\ev_{n,j}^+ - \ev_{n,i}^-) t \RB }{4 \lambda_{n+1} (\ev_{n,j}^+ - \ev_{n,i}^-) } \p_x \LB \LC a_{n+1}^+ \cdot a_{n+1}^- \RC \varphi_{n,i} \LC \ev_n^+ \RC \varphi_{n,j} \LC \ev_n^- \RC \RB \nonumber \\ 
& \qquad \ \  \LC Df(0) - \frac{\ev^+(0) + \ev^-(0)}{2} \id \RC \boldd
\end{align}}
\begin{equation}\label{eq est R_5}
\begin{aligned}
& \LV \rem^{(5),k,l}_{n+1} \RV & \lesssim_{M} & \ \frac{a_{n+1}^2 \lambda_n^2 |\nabla \boldu_n^{\delta_n}|}{\lambda_{n+1}} + \frac{|a_{n+1}| |\nabla a_{n+1}|}{\lambda_{n+1}}, \\
& \LV \nabla \rem^{(5),k,l}_{n+1} \RV & \lesssim_{M} & \ a_{n+1}^2 \LC \lambda_n^2 |\nabla \boldu_n^{\delta_n}| + \frac{\lambda_n^3 |\nabla \boldu_n^{\delta_n}|^2 + \lambda_n^2 |\nabla^2 \boldu_n^{\delta_n}|}{\lambda_{n+1}} \RC + \\
& & & \frac{\lambda_n^2 |a_{n+1}| |\nabla a_{n+1}| |\nabla \boldu_n^{\delta_n}|}{\lambda_{n+1}} + \frac{ |\nabla a_{n+1}|^2 + |a_{n+1}| |\nabla^2 a_{n+1}| }{\lambda_{n+1}}.
\end{aligned}
\end{equation}

\addtocontents{toc}{\protect\setcounter{tocdepth}{0}}
\section*{Acknowledgement} 
Robin Ming Chen is partially supported by the NSF grant: DMS 2205910. Alexis Vasseur is partially supported by the NSF grant: DMS 2306852.   Cheng Yu is is partially supported by the Collaboration Grants for Mathematicians from Simons Foundation.

\addtocontents{toc}{\setcounter{tocdepth}{1}} 
\bibliographystyle{siam}

\begin{thebibliography}{1}



\bibitem{BB}{\sc S.~ Bianchini and A.~ Bressan}, {\em Vanishing viscosity solutions of nonlinear hyperbolic systems,} Ann. of Math. (2)161(2005), no.1, 223–342.

\bibitem{BCM}
{\sc S.~Bianchini, R.~Colombo and F.~Monti}, {\em {$2\times 2$} systems of conservation laws with {$\bold
              L^\infty$} data}, J. Differential Equations, 249 (2010), pp.~3466--3488.



\bibitem{Br} {\sc A. ~Bressan,} {\em Hyperbolic systems of conservation laws. }  Oxford University Press, 2000.

\bibitem{BDle}{\sc A.~Bressan and C.~De~Lellis}, {\em A remark on the uniqueness of solutions
  to hyperbolic conservation laws}, Arch. Ration. Mech. Anal., 247 (2023),
  pp.~Paper No. 106, 12.

\bibitem{BG} {\sc A.~ Bressan and G.~ Guerra,} {\em Unique Solutions to Hyperbolic Conservation Laws with a Strictly Convex Entropy,} J. Differential Equations 387 (2024), 432–447.


\bibitem{BL} {\sc A.~ Bressan and M. Lewicka,} {\em A uniqueness condition for hyperbolic systems of conservation laws, } Discrete Cont. Dynam. Systems 6 (2000), 673–682.

\bibitem{BLe} {\sc A.~ Bressan and P. ~LeFloch,} {\em Uniqueness of weak solutions to systems of conservation
laws,}  Arch. Rational Mech. Anal. 140 (1997), no. 4, 301–317.



\bibitem{CKV} {\sc G.~ Chen, S.~ Krupa and A.~ Vasseur,} {\em Uniqueness and Weak-BV Stability for
$2\times 2$ Conservation Laws, } Arch. Ration. Mech. Anal. 246 (2022), no. 1, 299–332.

\bibitem{CKV2} {\sc G.~ Chen, M.-J.~ Kang and A.~ Vasseur,} {\em From Navier-Stokes to BV solutions of the barotropic Euler equations, } arXiv:2401.09305.



\bibitem{CDK} {\sc E. ~Chiodaroli, C. ~De Lellis and  O.~ Kreml,}{\em
Global ill-posedness of the isentropic system of gas dynamics,}
Comm. Pure Appl. Math. 68 (2015), no. 7, 1157-1190.



\bibitem{CL1}
{\sc M.~G. Crandall and P.-L. Lions}, {\em Viscosity solutions of
  hamilton-jacobi equations}, Transactions of the American Mathematical
  Society, 277 (1983), pp.~1--42.

\bibitem{Dafermos}
{\sc C.~M. Dafermos}, {\em Hyperbolic conservation laws in continuum physics},
  vol.~3, Springer, 2005.

\bibitem{Dafermos1}
{\sc C.~M. Dafermos}, {\em The second law of thermodynamics and stability},
Arch. Rational Mech. Anal., 70 (1979), no. 2, pp.~167--179.



\bibitem{de2009euler}
{\sc C.~De~Lellis and L.~Sz{\'e}kelyhidi~Jr}, {\em The {E}uler equations as a
  differential inclusion}, Annals of mathematics,  (2009), pp.~1417--1436.



\bibitem{de2010euler}
{\sc C.~De~Lellis and L.~Sz{\'e}kelyhidi~Jr}, {\em
On admissibility criteria for weak solutions of the Euler equations,} Arch. Ration. Mech. Anal., 195(1):225–260, 2010.


\bibitem{DiPerna}
{\sc R.~J. DiPerna}, {\em Uniqueness of solutions to hyperbolic conservation laws},
Indiana Univ. Math. J., 28 (1979), no. 1, pp.~137--188.


\bibitem{Giri-Kwon}
{\sc V.~Giri and H.~Kwon}, {\em On non-uniqueness of continuous entropy
  solutions to the isentropic compressible {E}uler equations}, Arch. Ration.
  Mech. Anal., 245 (2022), pp.~1213--1283.

\bibitem{G}{\sc J.~ Glimm,} {\em Solutions in the large for nonlinear hyperbolic systems of equations, } Comm. Pure Appl. Math. 18 (1965), 697-715.
\bibitem{GL} {\sc J.~ Glimm and P. ~Lax}, {\em Decay of solutions of systems of nonlinear hyperbolic conservation laws, }Amer. Math. Soc. Memoir 101 (1970).

\bibitem{GKV}{\sc W.~M.~Golding, S.~G.~Krupa and A.~F.~Vasseur}, {\em Sharp a-contraction estimates for small extremal shocks}, Journal of Hyperbolic Differential Equations,  20 (2023), no.3, pp.~541--602.


\bibitem{Johansson_2023}
{\sc C.~J.~P. Johansson and R.~Tione}, {\em T5 configurations and hyperbolic
  systems}, Communications in Contemporary Mathematics,  (2023).
  
  \bibitem{KV}{\sc M.-J.~ Kang and A.~ Vasseur,} {\em Uniqueness and stability of entropy shocks to the isentropic Euler system in a class of inviscid limits from a large family of Navier–Stokes systems}, Invent. Math., 224 (2021), pp.55--146.
  

\bibitem{K}
{\sc S. ~G. Krupa}, {\em Finite time BV blowup for Liu-admissible solutions to P-system via computer-assisted proof}, 	arXiv:2403.07784.
\bibitem{krupa2022nonexistence}
{\sc S.~G. Krupa and L.~Sz\'ekelyhidi~Jr}, {\em Nonexistence of {$T_4$}
  configurations for hyperbolic systems and the Liu entropy condition},
  (2022).




\bibitem{Liu} {\sc T.-P.~ Liu,}{\em Nonlinear stability of shock waves for viscous conservation laws,} Mem. Amer. Math. Soc. 328 (1985),
v+108.

\bibitem{LiuYang}{\sc T.P.~Liu and T.~Yang,}{\em {$L^1$} stability for {$2\times 2$} systems of hyperbolic conservation laws,} J. Amer. Math. Soc. 12 (1999), pp.~729-774.

\bibitem{Lorent-Peng}
{\sc A.~Lorent and G.~Peng}, {\em On the rank-1 convex hull of a set arising
  from a hyperbolic system of {L}agrangian elasticity}, Calc. Var. Partial
  Differential Equations, 59 (2020), pp.~Paper No. 156, 36.


\bibitem{R}{\sc B. ~Riemann,} {\em Uber die Fortpflanzung ebener Luftwellen von endlicher Schwingungsweite,} Gottingen Abh. Math. Cl. 8 (1860), pp.~43-65.

\bibitem{Tartar}{\sc L.~Tartar,}{\em Compensated compactness and applications to partial differential equations,} Nonlinear analysis and mechanics: {H}eriot-{W}att {S}ymposium,
              {V}ol. {IV}, Res. Notes in Math., 39 (1979), pp.~136-212.


\bibitem{V}{\sc A.~ Vasseur,} {\em Recent results on hydrodynamic limits. In Handbook of differential equations: evoluationary equations.} Vol. IV, Handb. Differ. Equ., pp.~323-376. Elsevier/North-Holland, Amsterdam, 2008.



\end{thebibliography}

\end{document}